\documentclass[11pt,reqno]{amsart}

\usepackage{amsmath,amssymb,amsthm}
\usepackage[T1]{fontenc}
\usepackage{graphicx}
\usepackage[letterpaper,margin=1.1in]{geometry}
\usepackage[hidelinks]{hyperref}
\hypersetup{pdftitle={Sharp entropic stability for Sidon sets and their generalizations},
  pdfkeywords={Sidon set, Bh[g] set, equivocation, list decoding, Renyi entropy, stability},
  pdfauthor={Jae Oh Woo}}

\newcommand{\IfTwoCol}[2]{#2}

\theoremstyle{plain}
\newtheorem{theorem}{Theorem}[section]
\newtheorem{corollary}[theorem]{Corollary}
\newtheorem{proposition}[theorem]{Proposition}
\newtheorem{lemma}[theorem]{Lemma}

\theoremstyle{definition}
\newtheorem{example}[theorem]{Example}

\newtheorem{question}[theorem]{Question}

\theoremstyle{remark}
\newtheorem{remark}[theorem]{Remark}

\newcommand{\Z}{\mathbb{Z}}
\newcommand{\PR}{\mathbb{P}}
\newcommand{\E}{\mathbb{E}}
\newcommand{\Uc}{\mathcal{U}}
\newcommand{\Ec}{\mathcal{E}}
\newcommand{\Cc}{\mathcal{C}}
\newcommand{\Rc}{\mathcal{R}}
\newcommand{\Lc}{\mathcal{L}}
\DeclareMathOperator{\supp}{supp}
\DeclareMathOperator{\Unif}{Unif}

\numberwithin{equation}{section}
\begin{document}

\title[{Entropic stability for Sidon sets}]
      {Sharp entropic stability for Sidon sets and their generalizations}

\author[Jae Oh Woo]{Jae Oh Woo$^{*}$}
\address{Amazon Web Services}
\thanks{$^{*}$This work was carried out independently of Amazon Web Services and
does not represent the views of Amazon Web Services.}
\email{jaeoh.woo@aya.yale.edu}

\date{}

\begin{abstract}
Let $X$ and $X'$ be independent copies of a random variable on an abelian group,
and suppose that the unordered pair $\{X,X'\}$ can almost be recovered from the sum
$X+X'$, in the sense that the equivocation $H(\{X,X'\}\mid X+X')$ is at most $C$.
Li, Gavalakis, and Kontoyiannis asked whether a bound $H(X)\le D$ alone then
forces almost all of the mass of $X$ onto a Sidon set, where recovery is exact. We show that the mass that must be deleted is, in the worst
case, $(1+o(1))D/\log(1/C)$ as $C\to0$, and that integer-valued examples attain
the leading constant $D$. The same holds for $h$-fold sums under a list-equivocation constraint with lists
of $g$ candidates, with constant $(h-1)D$. With R\'enyi orders $\alpha$ for the loss and $\beta$ for the
input constraint, stability holds exactly when $\beta\le1$ and $\alpha\ge\beta$.
For $\alpha\ge1$, we determine the exact crossover profile in the window
$1-\beta\asymp1/\log(1/C)$. For $\alpha\ge1$ and $\beta>1$
the deletion cost tends to an exact positive limit, and for $\alpha<1$ and
$\beta>\alpha$ stability fails by dilution. In the dilution regime, for unique recovery, budget
order two, and entropy budgets up to $\log4$, we determine the exact worst-case
deletion mass over the integers and show that it is strictly smaller than the
supremum over all abelian groups.
\end{abstract}

\maketitle

\section{Introduction}\label{sec:intro}

Suppose an observer sees the sum of two independent samples and
wants to recover the samples up to order. On the support $\{1,2,3\}$, the sum
$4$ leaves two possibilities, $\{1,3\}$ and $\{2,2\}$. Removing $2$ from the
support resolves the ambiguity. When the support is large, ambiguous sums may
occur with very small probability. How much input probability must be removed to
eliminate every ambiguity?

Let $X,X'$ be independent with the same discrete law on an abelian group $G$,
let $U=\{X,X'\}$ be their unordered pair, with repetitions retained, and put
$S=X+X'$. Exact recovery is possible precisely when the support of $X$ is a
\emph{Sidon set\nocorr}: each sum has a unique unordered representation. Finite
Sidon subsets of $\Z$ are, after translation, the Golomb rulers. The uncertainty
left after observing the sum is the equivocation $H(U\mid S)=H(U)-H(S)$, which
is exactly the deficit in the sharp entropic doubling inequality of Li,
Gavalakis, and Kontoyiannis \cite[Lemma~5.1]{LGK}
(Proposition~\ref{prop:basic}(4)). They bounded the deletion cost under a lower
bound on the smallest atom and asked whether an upper bound on $H(X)$ would
suffice \cite[after Example~5.4]{LGK}.

Write $\delta_2(X)$ for the infimum of $\PR(X\notin B)$ over Sidon subsets $B$
of the support. Let $\Phi(C,D)$ be the supremum of $\delta_2(X)$ over all
abelian groups and all laws with $H(U\mid S)\le C$ and $H(X)\le D$. All
logarithms are natural.

\begin{theorem}[Sharp entropic stability for Sidon sets]\label{thm:main}
For each fixed $D>0$,
\[
  \lim_{C\downarrow0}\Phi(C,D)\log(1/C)=D .
\]
The same limit holds when the supremum is restricted to integer-valued random
variables.
\end{theorem}

The upper bound follows by choosing a probability threshold $\tau$. Atoms below
the threshold carry at most $D/\log(1/\tau)$ mass, since each unit of their mass
contributes at least $\log(1/\tau)$ to $H(X)$. For each sum, keep one most likely
pair. By the Hellman--Raviv inequality \cite{HellmanRaviv}, all remaining pairs
together have probability at most $C/(2\log2)$. For each remaining pair whose two
atoms both have probability at least $\tau$, delete its lighter atom; if that
atom has mass $m$, the pair has probability at least $\tau m$. The total
deletion cost is therefore at most
\[
  \frac{D}{\log(1/\tau)}+\frac{C}{2\tau\log2},
\]
and $\tau=C\log^2(1/C)$ gives the upper bound (Corollary~\ref{cor:lgk}).

Why can the cost be this large? Distribute a small mass $q$ over $N$ separated
integer blocks and put the rest on one atom. Two samples can produce an ambiguous
sum only if both land in the same block, so for a fixed block with collisions
the equivocation is proportional to $q^2/N$. Yet each block must lose mass
before all its sums become unique, and we choose blocks for which the fraction
lost approaches one. The entropy budget permits $q\sim D/\log N$ as $N$ grows.
These two scales account for the logarithm and for its leading constant
(Section~\ref{sec:crit}).

\subsection{Which entropy constraints ensure stability?}\label{sec:diagram}

There are two choices to make when extending the Shannon result: how to measure
the information lost on observing the sum, and how to constrain the input law.
We use R\'enyi entropy \cite{Renyi} for both and allow the orders to differ.
Write $\alpha$ for the \emph{loss order} and $\beta$ for the \emph{budget
order\nocorr}, with $H_1=H$.

Let $U_h$ be the unordered multiset of $h$ independent copies of $X$ and $S_h$
their sum. A \emph{$B_h[g]$ set} allows at most $g$ unordered representations of
each sum; for $g=1$ these are the $B_h$ sets, and for $h=2$, $g=1$ the Sidon
sets. Exact list recovery on such a support means that the observer can always
give at most $g$ candidates, one of them correct, and a label
$\chi(U_h)\in\{1,\dots,g\}$ can then specify which candidate occurred. This
gives a measure of loss that vanishes precisely on $B_h[g]$ supports
(Proposition~\ref{prop:listzero}):
\begin{equation}\label{eq:defdelta}
  \Delta^{[g]}_{\alpha,h}(X)=
  \inf_{\chi}\bigl[H_\alpha(U_h)-H_\alpha\bigl(S_h,\chi(U_h)\bigr)\bigr].
\end{equation}
At $\alpha=1$ this is the least conditional entropy $H(U_h\mid S_h,\chi(U_h))$
over all labels, and for $g=1$ it is the equivocation $H(U_h\mid S_h)$. Let
$\Phi^{[g]}_{\alpha,\beta,D,h}(C)$ be the worst-case mass that must be deleted to
leave a $B_h[g]$ support under the constraints $\Delta^{[g]}_{\alpha,h}(X)\le C$
and $H_\beta(X)\le D$, the supremum again being over abelian groups and discrete
laws (Section~\ref{sec:setupdefs} gives the precise definition), and let
$\Phi^{\Z,[g]}$ be the version over $\Z$. Thus $\Phi(C,D)=\Phi_{1,1,D,2}(C)$.

\begin{theorem}[Two-order classification]\label{thm:phase}
Fix $h\ge2$, $g\ge1$ and $D>0$, and write $\Phi(C)=\Phi^{[g]}_{\alpha,\beta,D,h}(C)$.
Then $\Phi(C)\to0$ as $C\downarrow0$ exactly when $\beta\le1$ and
$\alpha\ge\beta$.
\begin{enumerate}
\item If $\beta<1$ and $\alpha\ge\beta$, then $\Phi(C)\asymp C^{\Theta_h(\alpha,\beta)}$
for small $C$, with constants depending only on $\alpha,\beta,D,h,g$, where
\begin{equation}\label{eq:theta}
  \Theta_h(\alpha,\beta)=\min\Bigl\{\frac1{h\alpha},
    \frac{1-\beta}{h\alpha-\beta}\Bigr\}.
\end{equation}
\item If $\beta=1$ and $\alpha\ge1$, then $\Phi(C)\log(1/C)\to(h\alpha-1)D$.
\item If $\beta>1$ and $\alpha\ge1$, put $q_*=1-e^{-(\beta-1)D/\beta}$. Then
$\Phi(C)\ge q_*$ for all $C>0$, and
$\Phi(C)-q_*=O\bigl(C^{\min\{1,(\beta-1)/(h\alpha-1)\}}\bigr)$ as $C\downarrow0$.
\item If $0<\alpha<1$ and $\beta>\alpha$, then $\inf_{C>0}\Phi(C)>0$. For $g=1$
the value at every $C>0$ is $1-m_\beta(D)$, where $m_\beta(D)$ is the infimum of
the largest atom among laws with $H_\beta\le D$.
\end{enumerate}
The first three conclusions and the positive lower bound in the fourth hold also
over the integers.
\end{theorem}

At $\alpha=\beta=1$ part 2 extends Theorem~\ref{thm:main} to $h$-fold sums and
lists of size $g$, with constant $(h-1)D$. Sharp means the optimal exponent in
part 1, the exact leading constant in part 2, the exact limit in part 3, and the
exact value, for $g=1$ and over all groups, in part 4 (Figure~\ref{fig:diagram}).

\IfTwoCol{\begin{figure}[t]}{\begin{figure*}[t]}
\centering
\IfTwoCol{\includegraphics[width=\columnwidth]{figures/phase-diagram-column.pdf}}%
        {\includegraphics[width=\textwidth]{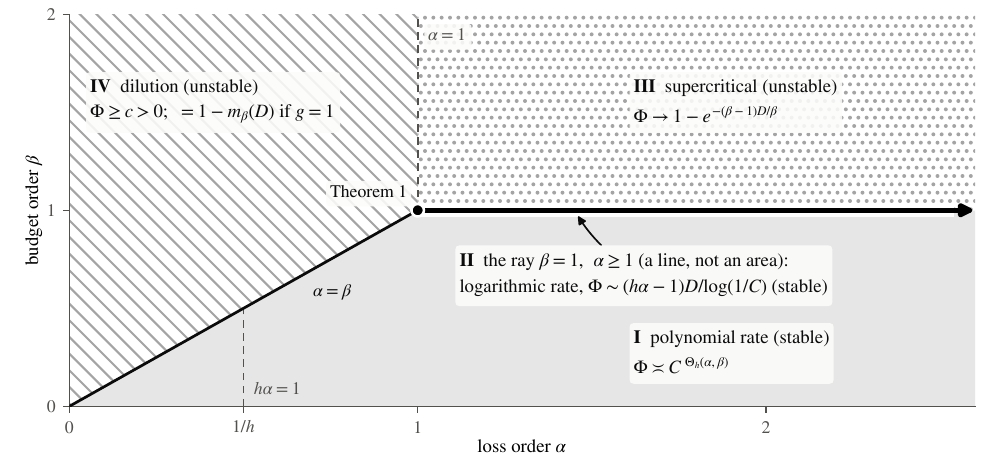}}
\caption{The two-order classification of Theorem~\ref{thm:phase}, for every fixed
$g$. Flat gray: the stable region $\{\beta\le1,\ \alpha\ge\beta\}$,
consisting of the polynomial region \textsc{i}, including its diagonal edge
$\alpha=\beta<1$, and the logarithmic ray \textsc{ii}, the half-line $\{\beta=1,\ \alpha\ge1\}$
(dark line with arrowhead), whose closed endpoint $(1,1)$ is Theorem~\ref{thm:main}. Textured: the two unstable regions, the
supercritical region \textsc{iii} (dotted) and the dilution region \textsc{iv} (hatched), separated
by the dashed line $\alpha=1$. The dashed segment $h\alpha=1$ changes only the
exponent $\Theta_h$ inside \textsc{i}; it is drawn for $h=2$. In \textsc{iv} the
value is $1-m_\beta(D)$ when $g=1$, and over $\Z$, for $g=1$, it is $1-e^{-D/2}$
when $\beta=2$ and $D\le\log4$.}\label{fig:diagram}
\IfTwoCol{\end{figure}}{\end{figure*}}

The two failure regimes have different causes. When $\beta>1$, the input can put
a fixed mass on increasingly many light atoms while respecting the budget. When
$\alpha<1$ and $\beta>\alpha$, even a vanishing mass can have a large effect:
adding many light atoms with distinct sums dilutes the measured loss while the
original support still requires deletion. Here the group can affect the exact
value. For instance, at $g=1$, $0<\alpha<1$, $\beta=2$ and $D=\log2$, the worst
deletion mass for every $C>0$ is $1-1/\sqrt2$ over the integers and $\frac12$ over
all abelian groups (Corollary~\ref{cor:Zbeta2}). For $1<\beta<h$ and sufficiently
large $D$, a uniform distribution on an integer interval forces strictly more
deletion than the heavy-atom construction (Proposition~\ref{prop:Zlarge}).

\subsection{The transition near the Shannon budget}\label{sec:transition}

For fixed $\alpha\ge1$, the polynomial exponent tends to zero as $\beta\uparrow1$.
The exponent alone therefore gives little guidance on how close $\beta$ must be
to one before the Shannon formula becomes accurate. The answer depends on the
observed loss $C$.

\begin{theorem}[Crossover at the Shannon budget]\label{thm:crossover}
Fix $h\ge2$, $g\ge1$, $\alpha\ge1$ and $D>0$, and put $k=h\alpha-1$. As
$C\downarrow0$ and $\beta\uparrow1$, suppose $(1-\beta)\log(1/C)/k\to
x\in(0,\infty)$. Then
\[
  \frac{\log(1/C)}{kD}\,\Phi^{[g]}_{\alpha,\beta,D,h}(C)\ \longrightarrow\
  \frac{x}{e^x-1}.
\]
The same limit holds over the integers.
\end{theorem}

At $x=1$, the worst deletion mass is asymptotically $1/(e-1)$, about $58\%$, of
the Shannon value at the same $C$. Section~\ref{sec:crossover} proves the formula
and describes its uniformity.

\subsection{Methods and relation to prior work}\label{sec:related}

Entropic approaches to sumsets relate the distribution of a sum to the structure
of its inputs \cite{Ruzsa,Tao,MMT,KM,GMT}; for R\'enyi orders see
\cite{MWWsperner,MWWpcg,WWM,WM}. Here the equality case specifies the target
support, and the aim is to reach that support by deleting as little probability
as possible. Our deficit differs from the entropic additive energy of Goh
\cite{Goh} in that its zero set is exactly the set of laws supported on
$B_h[g]$ sets, at every order. Classical Sidon extraction asks how many points can
be retained \cite{ErdosTuran,KSS,OBryant,JingMudgal}; see also
\cite{Cilleruelo,PachZakharov,BFR,BailleulRiblet,Shkredov}, and \cite{CMPSY} for a
different entropic approach to $B_h[g]$-type problems. The weighted problem
instead tracks the mass of the deleted points, which may vary even when the
support stays fixed.

For the upper bounds we may replace addition by any map on multisets. After each
output, retain its $g$ most likely inputs. A sharp coarsening inequality
(Theorem~\ref{thm:coarse}) controls the weight of the remaining inputs, and
deleting a lightest atom from each remaining multiset gives an admissible
support. Addition enters the matching examples. Separated blocks already appear
in \cite[Example~5.4]{LGK}. For general $h$ we use $B_{h-1}$ blocks from the
Bose--Chowla construction \cite{BoseChowla}: unique representation at order $h-1$
forces every collision at order $h$ to lie in a single block, which lets us
compute the loss and the deletion cost of the same law exactly.

Section~\ref{sec:prelim} fixes notation; Sections~\ref{sec:coarse}
and~\ref{sec:upper} develop the coarsening, deletion and entropy bounds;
Sections~\ref{sec:blocks} and~\ref{sec:lower} give the integer constructions and
the matching lower bounds; Section~\ref{sec:energy} treats finite sets and
representation counts, and Section~\ref{sec:open} lists open problems.

\section{Preliminaries}\label{sec:prelim}

\subsection{Definitions}\label{sec:setupdefs}
Let $h\ge2$ and $g\ge1$ be integers. For $A\subset G$ let $\Uc_h(A)$ be the
collection of multisets of size $h$ with elements in $A$, and $s(u)$ the sum of
$u$ counted with multiplicity. A set $B\subset G$ is a \emph{$B_h[g]$ set} if
every fibre of $s$ on $\Uc_h(B)$ has at most $g$ elements. Generalized Sidon sets
are classical \cite{Cilleruelo,OBryant}; the condition is the one defining the
$B_h^+[g]$ sets of \cite{JingMudgal}.

Throughout, $X$ is a discrete $G$-valued random variable with $p_a=\PR(X=a)$ and
countable support $A$, and $X_1,\dots,X_h$ are independent copies. We abbreviate
$\Uc_h=\Uc_h(A)$, put $U_h=\{X_1,\dots,X_h\}$, a random element of $\Uc_h$, and
$S_h=X_1+\dots+X_h=s(U_h)$, and write $w_u=\PR(U_h=u)$. For $\alpha>0$,
$\alpha\ne1$, write $M_\alpha(Y)=\sum_y\PR(Y=y)^\alpha$ and
$H_\alpha(Y)=\log M_\alpha(Y)/(1-\alpha)$, with $H_1=H$. The deficit
\eqref{eq:defdelta} is $\inf_\chi\bigl(H_\alpha(U_h)-H_\alpha((s,\chi)(U_h))\bigr)
\ge0$, the infimum over labellings $\chi:\Uc_h\to\{1,\dots,g\}$; for $\alpha\ne1$
it is a difference of two R\'enyi entropies, not a conditional entropy. Write
\begin{equation}\label{eq:dh}
  \delta_{h,g}(X):=1-\sup\bigl\{\PR(X\in B): B\subset A\text{ is }B_h[g]\bigr\},
\end{equation}
and $\delta_h=\delta_{h,1}$. For a second order $\beta>0$ set
\begin{equation}\label{eq:Phi}
\IfTwoCol{%
  \begin{aligned}
  \Phi^{[g]}_{\alpha,\beta,D,h}(C):=\sup_G\ \sup_{X\text{ on }G}
  \Bigl\{\delta_{h,g}(X):\ &\Delta^{[g]}_{\alpha,h}(X)\le C,\\
   &H_\beta(X)\le D\Bigr\},
  \end{aligned}}{%
  \Phi^{[g]}_{\alpha,\beta,D,h}(C):=\sup_G\ \sup_{X\text{ on }G}
  \Bigl\{\delta_{h,g}(X):\ \Delta^{[g]}_{\alpha,h}(X)\le C,\ H_\beta(X)\le D\Bigr\},}
\end{equation}
the inner supremum being over those $X$ for which the deficit is defined, that
is $M_\alpha(U_h)<\infty$ when $\alpha\ne1$ and $H(X)<\infty$ when $\alpha=1$;
the budget $H_\beta(X)\le D$ does not imply these conditions for all choices of
the two orders (for $\alpha>1$, $M_\alpha(U_h)\le1$ holds automatically). Every
upper bound below holds for each fixed $G$; $\Phi^{\Z,[g]}$ denotes the version
with $G=\Z$, and we drop the superscript $[g]$ when $g=1$. Finally let
\begin{equation}\label{eq:mbeta}
\IfTwoCol{%
  \begin{aligned}
    m_\beta(D):=\inf\bigl\{\lVert p\rVert_\infty:\ &p\text{ a probability
      vector},\\ &H_\beta(p)\le D\bigr\}\in(0,1]
  \end{aligned}}{%
  m_\beta(D):=\inf\bigl\{\lVert p\rVert_\infty:\ p\text{ a probability vector},\
  H_\beta(p)\le D\bigr\}\in(0,1]}
\end{equation}
be the least possible largest atom under the budget, computed in
Lemma~\ref{lem:mbeta}.

\subsection{Basic properties}

For $u\in\Uc_h$ with distinct elements $a_1,\dots,a_k$ of
multiplicities $m_1,\dots,m_k$ summing to $h$, $\nu_u=\binom{h}{m_1,\dots,m_k}$ is
the number of ordered $h$-tuples realizing $u$ and
\begin{equation}\label{eq:wdef}
  w_u=\PR(U_h=u)=\nu_u\prod_{i=1}^k p_{a_i}^{m_i} ,
\end{equation}
and we set $m_u=\min\{p_a:a\in u\}$, the minimum over the distinct elements of
$u$. We write $q_z=\PR(S_h=z)$, $P_r=P_r(X)=\sum_{a\in A}p_a^r$, and
$\eta(q)=-q\log q-(1-q)\log(1-q)$. All logarithms are natural. R\'enyi entropy
does not increase under a deterministic map, which is what makes
\eqref{eq:defdelta} nonnegative. For $h=2$ and $\alpha=1$ the deficit is
classical in disguise: since $H(X_1,X_2)=2H(X)$ and $U_2$ determines the ordered
pair up to a swap, which costs $\log2$ nats exactly when $X_1\ne X_2$,
\begin{equation}\label{eq:HU1}
  H(U_2)=2H(X)-(\log2)\Bigl(1-\sum_ap_a^2\Bigr),
\end{equation}
so $\Delta_{1,2}(X)$ is exactly the deficit in the sharp entropic doubling
inequality of \cite[Lemma 5.1]{LGK}, whose stability question is recalled in
Section~\ref{sec:intro}.

For any map $\pi$ on $\Uc_h$ put
$\Delta_{\alpha,\pi}(X):=H_\alpha(U_h)-H_\alpha\bigl(\pi(U_h)\bigr)$, and for a
map $f$ on $\Uc_h$ and $g\ge1$ define the \emph{$g$-list deficit}
\begin{equation}\label{eq:listdeficit}
  \Delta^{[g]}_{\alpha,f}(X):=\inf_{\chi:\Uc_h\to\{1,\dots,g\}}
   \Delta_{\alpha,(f,\chi)}(X)\ \ge\ 0 ,
\end{equation}
whenever the differences are well defined, which is the case under the
finiteness hypotheses of Proposition~\ref{prop:basic}. Thus
$\Delta^{[g]}_{\alpha,s}=\Delta^{[g]}_{\alpha,h}$ is the deficit
\eqref{eq:defdelta}, and $\Delta_{\alpha,h}=\Delta^{[1]}_{\alpha,h}
=H_\alpha(U_h)-H_\alpha(S_h)$.

\emph{Finite sets.} For finite $A\subset G$ with $\lvert A\rvert=n$ let
$r_h(z)=\sum_{u:s(u)=z}\nu_u$ be the number of ordered representations of $z$ as
an $h$-fold sum from $A$, and let $b_{h,g}(A)$ be the largest size of a $B_h[g]$
subset of $A$, with $b_h=b_{h,1}$. For $\alpha>0$, $\alpha\ne1$, the
\emph{$g$-split moment gap} is
\begin{equation}\label{eq:Rg}
  \Rc^{[g]}_{\alpha,h}(A):=\inf_{\chi}\Bigl\lvert
   \sum_{u\in\Uc_h(A)}\nu_u^\alpha
   -\sum_{(z,c)}\Bigl(\sum_{\substack{u\,:\,s(u)=z\\ \chi(u)=c}}\nu_u
     \Bigr)^{\!\alpha}\Bigr\rvert ,
\end{equation}
the infimum running over all labellings $\chi:\Uc_h(A)\to\{1,\dots,g\}$. For
$g=1$ it is $\Rc_{\alpha,h}(A)=\bigl\lvert\sum_u\nu_u^\alpha-\sum_zr_h(z)^\alpha
\bigr\rvert$, the distance between the $\alpha$-th moment of the representation
function and the value it takes on a $B_h$ set. For $\alpha\ne1$ it vanishes
exactly when $A$ is a $B_h[g]$ set (Lemma~\ref{lem:blockconst}); at $\alpha=1$
both sums equal $n^h$ for every $A$, so the restriction $\alpha\ne1$ is needed,
and the Shannon functional \eqref{eq:LambdaF} takes its place.

\begin{proposition}\label{prop:MU}
For every $\alpha>0$ and $h\ge2$,
\[
  M_\alpha(U_h)\le
  \begin{cases}
    P_\alpha^{\,h}, & 0<\alpha<1,\\
    (h!)^{\alpha-1}P_\alpha^{\,h}, & \alpha>1,
  \end{cases}
\]
and for $h=2$ one has the identity
$M_\alpha(U_2)=2^{\alpha-1}P_\alpha^2+(1-2^{\alpha-1})P_{2\alpha}$.
\end{proposition}

\begin{proof}
By \eqref{eq:wdef}, $1\le\nu_u\le h!$ and $u$ corresponds to exactly $\nu_u$
ordered $h$-tuples, so
$M_\alpha(U_h)=\sum_{u\in\Uc_h}\nu_u^{\alpha-1}\,\nu_u\prod_ip_{a_i}^{\alpha m_i}$.
For $\alpha<1$ bound $\nu_u^{\alpha-1}\le1$; the remaining sum runs over ordered
$h$-tuples and equals $P_\alpha^h$. For $\alpha>1$ bound
$\nu_u^{\alpha-1}\le(h!)^{\alpha-1}$. The case $h=2$ follows by splitting
$\Uc_2$ into off-diagonal and diagonal pairs and using
$2\sum_{a<b}p_a^\alpha p_b^\alpha=P_\alpha^2-P_{2\alpha}$.
\end{proof}

\begin{proposition}\label{prop:basic}
Let $\alpha>0$ and $h\ge2$, and let $\pi$ be any map on $\Uc_h$. Assume that
$M_\alpha(U_h)<\infty$ if $\alpha\ne1$, and that $H(X)<\infty$ if $\alpha=1$, so
that $\Delta_{\alpha,\pi}(X)$ is well defined. Then:
\begin{enumerate}
\item $\Delta_{\alpha,\pi}(X)\ge0$, with equality if and only if $\pi$ is
injective on $\Uc_h$.
\item If $\pi'=\psi\circ\pi$ for some map $\psi$, then
$\Delta_{\alpha,\pi}(X)\le\Delta_{\alpha,\pi'}(X)$: further merging can only
increase the loss.
\item If $\alpha\ne1$ then
$M_\alpha(U_h)=M_\alpha(\pi(U_h))\,e^{(1-\alpha)\Delta_{\alpha,\pi}(X)}$; in
particular $M_\alpha(\pi(U_h))\le M_\alpha(U_h)$ for $\alpha<1$ and
$M_\alpha(\pi(U_h))\ge M_\alpha(U_h)$ for $\alpha>1$.
\item Taking $\pi=s$, $g=1$, the quantity $\Delta_{1,2}(X)$ equals the deficit
of \emph{\cite[Lemma 5.1]{LGK}}, namely
$H(X)-(\log2)\bigl(1-\sum_ap_a^2\bigr)-\bigl(H(X+X')-H(X)\bigr)$.
\end{enumerate}
\end{proposition}

\begin{proof}
(1) Group $\Uc_h$ into the fibres of $\pi$. For $0<\alpha<1$ subadditivity of
$t\mapsto t^\alpha$ gives $\PR(\pi(U_h)=z)^\alpha\le\sum_{u\mapsto z}w_u^\alpha$
and hence $M_\alpha(\pi(U_h))\le M_\alpha(U_h)$; since $1/(1-\alpha)>0$ this
gives $\Delta_{\alpha,\pi}\ge0$. For $\alpha>1$ the map is superadditive and
$1/(1-\alpha)<0$. Under the finiteness hypothesis the comparison may be made
term by term, and equality forces every fibre to be a singleton. For $\alpha=1$
and $H(X)<\infty$ we have $H(U_h)\le hH(X)<\infty$ and
$\Delta_{1,\pi}=H(U_h\mid\pi(U_h))$, which vanishes iff $U_h$ is a function of
$\pi(U_h)$.

(2) $\pi'(U_h)=\psi(\pi(U_h))$ is a function of $\pi(U_h)$, so
$H_\alpha(\pi'(U_h))\le H_\alpha(\pi(U_h))$.

(3) Immediate from the definition of $H_\alpha$. (4) Substitute \eqref{eq:HU1}
into $H(U_2)-H(S_2)$ and use $H(S_2)=H(X)+\bigl(H(X+X')-H(X)\bigr)$.
\end{proof}

\begin{proposition}\label{prop:listzero}
Under the hypotheses of Proposition~\ref{prop:basic}, and for every $g\ge1$,
\IfTwoCol{%
\[
  \begin{aligned}
  \Delta^{[g]}_{\alpha,f}(X)=0\iff\ &\text{every fibre of }f\text{ meets}\\
   &\Uc_h\text{ in at most }g\text{ multisets}.
  \end{aligned}
\]}{%
\[
  \Delta^{[g]}_{\alpha,f}(X)=0
  \iff \text{every fibre of }f\text{ meets }\Uc_h\text{ in at most }g
  \text{ multisets}.
\]}
In particular $\Delta^{[g]}_{\alpha,h}(X)=0$ if and only if $\supp X$ is a
$B_h[g]$ set; for $g=1$ this is the equality case of \eqref{eq:defdelta}.
\end{proposition}

\begin{proof}
If every fibre has at most $g$ elements, choose $\chi$ injective on each fibre;
then $(f,\chi)$ is injective on $\Uc_h$ and $\Delta_{\alpha,(f,\chi)}(X)=0$ by
Proposition~\ref{prop:basic}(1), so the infimum is $0$.

Conversely, suppose some fibre contains $g+1$ distinct multisets
$u_0,\dots,u_g$. For every $\chi$ two of them, say $u_i$ and $u_j$, receive the
same label, so $(f,\chi)$ takes the same value at $u_i$ and $u_j$. Let
$\pi_{ij}$ be the map on $\Uc_h$ that identifies $u_i$ with $u_j$ and is
injective elsewhere. Then $(f,\chi)$ factors through $\pi_{ij}$, so
Proposition~\ref{prop:basic}(2) gives
$\Delta_{\alpha,(f,\chi)}(X)\ge\Delta_{\alpha,\pi_{ij}}(X)$. Hence
\[
  \Delta^{[g]}_{\alpha,f}(X)\ \ge\
  \min_{0\le i<j\le g}\Delta_{\alpha,\pi_{ij}}(X)\ >\ 0 ,
\]
the minimum being over finitely many pairs and each term positive by
Proposition~\ref{prop:basic}(1). The statement for $f=s$ is
the definition of a $B_h[g]$ set.
\end{proof}

Two elementary facts about the budget alone make the region $\alpha<\beta$ of
Theorem~\ref{thm:phase} accessible. We first characterize the quantity
$m_\beta(D)$ of \eqref{eq:mbeta}.

\begin{lemma}[the least possible largest atom]\label{lem:mbeta}
For $m\in(0,1]$ let $n=\lfloor1/m\rfloor$ and let $v_m$ be the vector with $n$
entries equal to $m$ followed by one entry $1-nm$ (the last entry omitted when
$nm=1$). Then $v_m$ majorizes every probability vector $p$ with
$\lVert p\rVert_\infty\le m$; the same partial-sum comparison proves
$v_m\succ p$ for countably supported $p$ as well. Hence
\[
  \min\bigl\{H_\beta(p):\lVert p\rVert_\infty\le m\bigr\}=H_\beta(v_m) ,
\]
the map $m\mapsto H_\beta(v_m)$ is continuous and strictly decreasing from
$+\infty$ to $0$ on $(0,1]$, and $m_\beta(D)$ is the unique $m$ with
$H_\beta(v_m)=D$; in particular the infimum in \eqref{eq:mbeta} is attained, at
$v_{m_\beta(D)}$, and $D\mapsto m_\beta(D)$ is continuous and strictly
decreasing. For $0<D<\log2$ one has $n=1$, and $m_\beta(D)$ is then the unique
$m\in(\tfrac12,1)$ with $H_\beta(m,1-m)=D$; at $D=\log2$ one has
$m_\beta(\log2)=\tfrac12$, attained by $v_{1/2}=(\tfrac12,\tfrac12)$.
\end{lemma}

\begin{proof}
If $\lVert p\rVert_\infty\le m$ then the decreasing rearrangement of $p$ has
$\sum_{i\le k}p_{(i)}\le km$ for every $k\le n$, while the partial sums of $v_m$
are exactly $km$ for $k\le n$ and $1$ afterwards; so $v_m\succ p$. R\'enyi
entropy is Schur concave for every order \emph{\cite{MOA}}, whence the displayed
minimum. Continuity of $m\mapsto H_\beta(v_m)$ is clear on each interval
$[\frac1{n+1},\frac1n]$, and at an endpoint $m=1/n$ both one-sided limits equal
$H_\beta$ of the uniform vector on $n$ points; strict monotonicity holds because
decreasing $m$ strictly refines the majorization order. Finally $H_\beta(v_1)=0$
and $H_\beta(v_m)\ge\log\lfloor1/m\rfloor\to\infty$ as $m\downarrow0$. The last
sentence is the case $n=1$, together with $v_{1/2}=(\frac12,\frac12)$.
\end{proof}

\begin{proposition}[a universal ceiling]\label{prop:ceiling}
For every $h\ge2$, $g\ge1$, $\alpha,\beta,D>0$ and $C>0$,
\IfTwoCol{%
\[
  \begin{aligned}
  &\delta_{h,g}(X)\ \le\ 1-\lVert p\rVert_\infty ,\\
  &\Phi^{[g]}_{\alpha,\beta,D,h}(C)\ \le\ 1-m_\beta(D) .
  \end{aligned}
\]}{%
\[
  \delta_{h,g}(X)\ \le\ 1-\lVert p\rVert_\infty
  \qquad\text{and hence}\qquad
  \Phi^{[g]}_{\alpha,\beta,D,h}(C)\ \le\ 1-m_\beta(D) .
\]}
\end{proposition}

\begin{proof}
A singleton is a $B_h[g]$ set, so taking $B=\{a\}$ with $p_a=\lVert
p\rVert_\infty$ in \eqref{eq:dh} gives the first bound; the second is
\eqref{eq:mbeta}, since $H_\beta(X)\le D$.
\end{proof}

\section{Sharp coarsening and deletion}\label{sec:coarse}

The upper bounds in Theorem~\ref{thm:phase} use no additive structure, and we
state them for an arbitrary map. Let $f$
be \emph{any} map on $\Uc_h$ and put
\begin{equation}\label{eq:Lg}
  \Lc_{\alpha,g}(f):=\inf_{\chi:\Uc_h\to\{1,\dots,g\}}
   \Bigl\lvert\,\sum_{u}w_u^\alpha
   -\sum_{(z,c)}\Bigl(\sum_{\substack{u\,:\,f(u)=z\\ \chi(u)=c}}w_u
     \Bigr)^{\!\alpha}\Bigr\rvert ,
\end{equation}
the least $\alpha$-power-sum loss incurred by observing $f$ together with $g$
auxiliary labels. Call $B\subset A$ \emph{$(f,g)$-admissible} if every fibre of
$f$ meets $\Uc_h(B)$ in at most $g$ multisets, and set
$\delta_{f,g}(X)=1-\sup\{\PR(X\in B):B\text{ is }(f,g)\text{-admissible}\}$; for
$f=s$ the admissible sets are the $B_h[g]$ subsets of $A$. The argument is the
chain
\begin{equation}\label{eq:chain}
\IfTwoCol{%
\boxed{\ \begin{array}{c}
  \text{small loss}\ \Rightarrow\ \text{small excess mass}\\
  \Rightarrow\ \text{small deletion}\ \Rightarrow\ \text{fibres of size}\le g
\end{array}\ }}{%
\boxed{\ \text{small loss}\ \Rightarrow\ \text{small excess mass}
  \ \Rightarrow\ \text{small deletion}\ \Rightarrow\ \text{fibres of size}\le g\ }}
\end{equation}
whose first arrow is the following inequality.

\begin{theorem}[{Sharp coarsening inequality with $g$ auxiliary labels}]\label{thm:coarse}
Let $(w_u)_{u\in W}$ be a summable family of positive reals on a countable set
$W$, let $\pi:W\to Z$ be any map, and for each $z$ in the image let
$w_{z,1}\ge w_{z,2}\ge\cdots$ be the weights in the fibre $\pi^{-1}(z)$
arranged in nonincreasing order. Fix $g\ge1$, put
\begin{equation}\label{eq:T}
  T_{\alpha,g}(\pi):=\sum_z\ \sum_{j>g}w_{z,j}^{\,\alpha}\qquad(\alpha>0),
\end{equation}
and for a labelling $\lambda:W\to\{1,\dots,g\}$ let $\sigma_\lambda(u)$ be the
total weight of $\{u':\pi(u')=\pi(u),\ \lambda(u')=\lambda(u)\}$.
\begin{enumerate}
\item \emph{(R\'enyi orders.)} Let $\alpha>0$, $\alpha\ne1$, put
$d_\alpha=\lvert2-2^\alpha\rvert$ and assume $\sum_uw_u^\alpha<\infty$. Then for
\emph{every} labelling $\lambda$,
\begin{equation}\label{eq:collision}
  \Bigl\lvert\sum_{u\in W}w_u^\alpha
   -\sum_{(z,c)}\Bigl(\sum_{\substack{u\,:\,\pi(u)=z\\ \lambda(u)=c}}w_u
   \Bigr)^{\!\alpha}\Bigr\rvert\ \ge\ d_\alpha\,T_{\alpha,g}(\pi) .
\end{equation}
\item \emph{(Shannon order.)} For every labelling $\lambda$,
\begin{equation}\label{eq:collisionshannon}
  \sum_{u\in W}w_u\log\frac{\sigma_\lambda(u)}{w_u}\ \ge\ (2\log2)\,
  T_{1,g}(\pi) .
\end{equation}
\end{enumerate}
The constants $d_\alpha$ and $2\log2=\lim_{\alpha\to1}d_\alpha/\lvert1-\alpha\rvert$
are optimal for every $g\ge1$; equality holds in both when one fibre consists of
$g+1$ atoms of equal weight, two of which share a label, and all other fibres
are singletons. For $g=1$ equality holds in \eqref{eq:collision} if and only if
every nonsingleton fibre of $\pi$ consists of exactly two elements of equal
weight.
\end{theorem}

In words, allowing $g$ output labels leaves an unavoidable loss proportional,
with the best possible constant, to the mass of all but the $g$ heaviest atoms
in each fibre. When the weights are a probability law, $W=\Uc_h$ and $\pi=s$, the
left side of \eqref{eq:collisionshannon} is $H(U_h\mid S_h,\lambda(U_h))$, and
$T_{1,g}(s)$ is the probability that $U_h$ is not among the $g$ most likely
multisets with its sum, the error probability of the best list decoder of size
$g$. For $g=1$, part 2 is therefore the Hellman--Raviv inequality
$P_e\le H(U\mid S)/(2\log2)$ \cite{HellmanRaviv}. What the deletion argument
needs is its list form with labels inside the conditioning, and its R\'enyi
analogues, in which the constant changes from $2\log2$ to $\lvert2-2^\alpha\rvert$.
Part 1 is a quantitative form of the Schur concavity of $t\mapsto\sum t_i^\alpha$
under merging of coordinates \cite{MOA}. The right-hand sides do not depend on
$\lambda$, so they survive the infimum over labellings, and they exempt exactly
the $g$ heaviest multisets of each fibre, which are the $g$ representations that
survive the deletion step.

The second arrow in \eqref{eq:chain} deletes a lightest point from every excess
multiset (Lemma~\ref{lem:deletion}), and the third is where $h$ enters: a
multiset whose points all have weight at least $m$ has weight at least $m^h$, and
$t\mapsto t^{h\alpha}$ is subadditive exactly when $h\alpha\le1$.

\begin{theorem}[Deletion principle]\label{thm:rawmoment}
Let $h\ge2$, $g\ge1$, let $f$ be any map on $\Uc_h$, let $0<\alpha\le1/h$, and
suppose $\sum_uw_u^\alpha<\infty$. Then, with $d_\alpha=\lvert2-2^\alpha\rvert$,
\[
  \delta_{f,g}(X)\ \le\
  \Bigl(\frac{\Lc_{\alpha,g}(f)}{d_\alpha}\Bigr)^{1/(h\alpha)} .
\]
\end{theorem}

The theorem requires neither an entropy constraint nor a group structure, and the
hypothesis $h\alpha\le1$ is exactly what the proof needs, so $h\alpha=1$ is the
critical boundary of the weighted removal problem without a budget. For uniform
weights on a finite set the situation is simpler: every multiset then has weight
at least the $h$-th power of every atom, and a linear bound holds at every order
(Proposition~\ref{prop:finiteall}). Because admissibility passes to subsets, the
principle also handles several notions of collision at once, at the cost of
adding the separate deletions.

\begin{corollary}[several maps at once]\label{cor:multimap}
Let $f_1,\dots,f_k$ be maps on $\Uc_h$ and $0<\alpha\le1/h$, with
$\sum_uw_u^\alpha<\infty$. Then some $B\subset A$ is $(f_i,g)$-admissible for
every $i$ and has
\[
  \PR(X\notin B)\ \le\ \sum_{i=1}^{k}
  \Bigl(\frac{\Lc_{\alpha,g}(f_i)}{d_\alpha}\Bigr)^{1/(h\alpha)} .
\]
In particular, if $X$ takes values in $\Z_{>0}$ and $h=2$, $g=1$, $\alpha=\frac12$,
then some $B$ that is Sidon and also multiplicatively Sidon, in that $ab=cd$ with
$a,b,c,d\in B$ forces $\{a,b\}=\{c,d\}$, has
$\PR(X\notin B)\le(\Lc^{+}+\Lc^{\times})/(2-\sqrt2)$, where $\Lc^{+}$ and
$\Lc^{\times}$ are the losses \eqref{eq:Lg} of the sum and of the product.
\end{corollary}

Additive and multiplicative Sidon subsets have been studied in the cardinality
setting \cite{JingMudgal}; Corollary~\ref{cor:multimap} is a weighted removal
statement and is not comparable with such worst-case bounds. Its point is that
the passage from one map to several, like the passage from $g=1$ to general $g$,
requires no new argument and preserves the removal exponent, apart from the
additive union-bound cost.

\begin{proof}[Proof of Theorem~\ref{thm:coarse}]
\emph{Step 1: two atoms.} For $a\ge b>0$ and $\alpha>0$, $\alpha\ne1$,
\begin{equation}\label{eq:twoatom}
  \bigl\lvert a^\alpha+b^\alpha-(a+b)^\alpha\bigr\rvert\ \ge\ d_\alpha b^\alpha .
\end{equation}
Both sides are homogeneous of degree $\alpha$, so we may take $b=1$ and write
$t=a\ge1$. For $0<\alpha<1$ the left side equals
$\psi(t):=t^\alpha+1-(t+1)^\alpha$, which is nonnegative by subadditivity and
satisfies
$\psi'(t)=\alpha\bigl(t^{\alpha-1}-(t+1)^{\alpha-1}\bigr)>0$
because $\alpha-1<0$; hence $\psi(t)\ge\psi(1)=2-2^\alpha=d_\alpha$. For
$\alpha>1$ the left side equals $\varphi(t):=(t+1)^\alpha-t^\alpha-1$, with
$\varphi'(t)=\alpha\bigl((t+1)^{\alpha-1}-t^{\alpha-1}\bigr)>0$, so
$\varphi(t)\ge\varphi(1)=2^\alpha-2=d_\alpha$.

\emph{Step 2: one class.} Let $P$ be a nonempty countable index set carrying
positive weights with finite sum $\sigma$ and finite $\alpha$-power sum. Because
only finitely many weights exceed any positive threshold, they can be arranged
in nonincreasing order $x_1\ge x_2\ge\cdots$. We claim
\begin{equation}\label{eq:oneclass}
  \Bigl\lvert\sum_{i}x_i^\alpha-\sigma^\alpha\Bigr\rvert
  \ \ge\ d_\alpha\sum_{i\ge2}x_i^\alpha .
\end{equation}
Put $a_k=x_1+\dots+x_k$. Since $a_k\ge x_1\ge x_{k+1}$, \eqref{eq:twoatom}
applies to the pair $(a_k,x_{k+1})$ for every $k\ge1$. For $\alpha<1$ each of
the quantities $a_k^\alpha+x_{k+1}^\alpha-a_{k+1}^\alpha$ is therefore at least
$d_\alpha x_{k+1}^\alpha$, and summing over $k=1,\dots,K-1$ telescopes to
\[
  \sum_{i\le K}x_i^\alpha-a_K^\alpha
  \ \ge\ d_\alpha\!\!\sum_{2\le i\le K}\!\!x_i^\alpha .
\]
Letting $K\to\infty$ and using $a_K\uparrow\sigma$ gives \eqref{eq:oneclass}.
For $\alpha>1$ the same computation applies with all three signs reversed.

\emph{Step 3: one fibre.} Fix $z$ in the image of $\pi$ and let
$x_1\ge x_2\ge\cdots$ be the weights in $\pi^{-1}(z)$. A labelling $\lambda$
splits $\pi^{-1}(z)$ into nonempty classes $P_1,\dots,P_{g'}$ with $g'\le g$ and
sums $\sigma_1,\dots,\sigma_{g'}$. By Proposition~\ref{prop:basic}(1) applied
classwise, equivalently by sub- respectively superadditivity of
$t\mapsto t^\alpha$, the numbers
$\sum_{u\in P_c}w_u^\alpha-\sigma_c^\alpha$ all have the same sign, so
\IfTwoCol{%
\[
  \begin{aligned}
  \Bigl\lvert\sum_{u\in\pi^{-1}(z)}w_u^\alpha
   -\sum_{c=1}^{g'}\sigma_c^\alpha\Bigr\rvert
  &=\sum_{c=1}^{g'}\Bigl\lvert\sum_{u\in P_c}w_u^\alpha-\sigma_c^\alpha
    \Bigr\rvert\\
  &\ \ge\ d_\alpha\sum_{c=1}^{g'}\ \sum_{u\in P_c\setminus\{\mu_c\}}
    w_u^\alpha,
  \end{aligned}
\]}{%
\[
  \Bigl\lvert\sum_{u\in\pi^{-1}(z)}w_u^\alpha
   -\sum_{c=1}^{g'}\sigma_c^\alpha\Bigr\rvert
  =\sum_{c=1}^{g'}\Bigl\lvert\sum_{u\in P_c}w_u^\alpha-\sigma_c^\alpha
    \Bigr\rvert
  \ \ge\ d_\alpha\sum_{c=1}^{g'}\ \sum_{u\in P_c\setminus\{\mu_c\}}w_u^\alpha,
\]}
by \eqref{eq:oneclass}, where $\mu_c$ denotes a fixed element of $P_c$ of maximal
weight, chosen arbitrarily if there are ties, which
exists for the reason given in Step 2. The elements $\mu_1,\dots,\mu_{g'}$ are
distinct and $g'\le g$, so $\sum_cw_{\mu_c}^\alpha\le\sum_{j\le g}x_j^\alpha$
because $t\mapsto t^\alpha$ is increasing. Hence the last display is at least
$d_\alpha\sum_{j>g}x_j^\alpha$.

\emph{Step 4: summation.} All fibrewise defects have the same sign, so by
Tonelli's theorem the left side of \eqref{eq:collision} is the sum over $z$ of
the quantities estimated in Step 3, and \eqref{eq:collision} follows. The
finiteness hypothesis makes the left side an absolutely convergent difference:
for $\alpha<1$ subadditivity bounds the second sum by
$\sum_uw_u^\alpha<\infty$, and for $\alpha>1$ both sums are finite because the
weights are summable.

\emph{Step 5: the Shannon order.} Write $\kappa(x)=x\log x$, with $\kappa(0)=0$.
For $a\ge b>0$,
\begin{equation}\label{eq:twoatomshannon}
  \kappa(a+b)-\kappa(a)-\kappa(b)\ \ge\ (2\log2)\,b .
\end{equation}
Indeed, with $t=a/b\ge1$ the left side is $b\,\varphi_1(t)$, where
$\varphi_1(t)=(t+1)\log(t+1)-t\log t$ has $\varphi_1'(t)=\log\frac{t+1}{t}>0$, so
$\varphi_1(t)\ge\varphi_1(1)=2\log2$. For a class $P$ with weights
$x_1\ge x_2\ge\cdots$, total $\sigma$ and partial sums $a_k=x_1+\dots+x_k$, the
contribution of $P$ to the left side of \eqref{eq:collisionshannon} is
$\sum_{u\in P}w_u\log(\sigma/w_u)=\kappa(\sigma)-\sum_i\kappa(x_i)$, a series of
nonnegative terms; telescoping as in Step~2, with \eqref{eq:twoatomshannon}
applied to $(a_k,x_{k+1})$, gives
$\kappa(a_K)-\sum_{i\le K}\kappa(x_i)\ge(2\log2)\sum_{2\le i\le K}x_i$, and
letting $K\to\infty$, by continuity of $\kappa$ and monotone convergence,
$\kappa(\sigma)-\sum_i\kappa(x_i)\ge(2\log2)\sum_{i\ge2}x_i$. Steps 3 and 4 then
apply verbatim with exponent $1$ on the weights, the summands now being
nonnegative by construction. Finally,
$d_\alpha/\lvert1-\alpha\rvert=\lvert2^\alpha-2\rvert/\lvert\alpha-1\rvert\to
\frac{d}{d\alpha}2^\alpha\big|_{\alpha=1}=2\log2$.

\emph{Optimality.} Fix $g\ge1$ and $\varepsilon\in(0,\frac12)$, and let $W$
carry one fibre of $\pi$ with weights $1,\dots,1$ ($g-1$ of them) together with
two weights equal to $\varepsilon$, all other fibres being singletons; rescale so
that the total is $1$. Among labellings by $g$ colours at least one pair of the
$g+1$ atoms of that fibre must share a colour. Merging the two atoms of weight
$\varepsilon$ costs exactly $d_\alpha\varepsilon^\alpha$ by the equality case of
\eqref{eq:twoatom}, whereas merging an atom of weight $1$ with one of weight
$\varepsilon$ costs $\lvert1+\varepsilon^\alpha-(1+\varepsilon)^\alpha\rvert$
and merging two atoms of weight $1$ costs $d_\alpha$, both of which exceed
$d_\alpha\varepsilon^\alpha$ for small $\varepsilon$; so the optimal labelling
merges exactly the two light atoms. Since the sorted fibre is
$1,\dots,1,\varepsilon,\varepsilon$, we have
$T_{\alpha,g}(\pi)=\varepsilon^\alpha$, and \eqref{eq:collision} is an equality.
The same family settles the Shannon order: merging the two light atoms costs
exactly $(2\log2)\varepsilon$ by the equality case $t=1$ of
\eqref{eq:twoatomshannon}, merging an atom of weight $1$ with one of weight
$\varepsilon$ costs $\kappa(1+\varepsilon)-\kappa(\varepsilon)
\ge\varepsilon\log(1/\varepsilon)$, and merging two atoms of weight $1$ costs
$2\log2$, so for small $\varepsilon$ the optimal labelling again merges the two
light atoms and \eqref{eq:collisionshannon} is an equality with
$T_{1,g}(\pi)=\varepsilon$. For a fibre of $g+1$ atoms of equal weight $x$ and a
labelling that gives exactly two of them the same label, both sides of
\eqref{eq:collision} equal $d_\alpha x^\alpha$ and both sides of
\eqref{eq:collisionshannon} equal $(2\log2)x$.

\emph{Equality for $g=1$.} Here $\lambda$ is constant, Step 3 is lossless, and
equality in \eqref{eq:collision} means equality in \eqref{eq:oneclass} for every
nonsingleton fibre. Equality there forces equality at each telescoped step, that
is $a_k=x_{k+1}$ for every $k$ occurring. From $a_1=x_1$ we get $x_1=x_2$; if a
third element were present we would need $x_3=a_2=2x_1$, contradicting
$x_3\le x_1$. So the fibre has exactly two elements of equal weight; conversely
such a fibre gives $\lvert2(\sigma/2)^\alpha-\sigma^\alpha\rvert
=d_\alpha(\sigma/2)^\alpha$, an equality.
\end{proof}

For the rest of this section fix a map $f$ on $\Uc_h$ and a total order on
$\Uc_h$, and let
\begin{equation}\label{eq:Ig}
\IfTwoCol{%
\begin{split}
  I_g=I_g(f):=\bigl\{u\in\Uc_h:\ &u\text{ is not among the }g\text{ heaviest}\\
   &\text{elements of its }f\text{-fibre}\bigr\},
\end{split}}{%
  I_g=I_g(f):=\bigl\{u\in\Uc_h:\ u\text{ is not among the }g\text{ heaviest
   elements of its }f\text{-fibre}\bigr\},}
\end{equation}
ties in the weights being broken by the fixed order. Thus
$\sum_{u\in I_g}w_u^\alpha=T_{\alpha,g}(f)$ and
$\sum_{u\in I_g}w_u=\sum_z\sum_{j>g}w_{z,j}$.

\begin{lemma}[Deletion]\label{lem:deletion}
Let $\mathcal A\subset A$ and
$I_{g,\mathcal A}=\{u\in I_g:u\subset\mathcal A\}$. Then $w_u\ge m_u^{\,h}$ for
every $u\in\Uc_h$, and there is an $(f,g)$-admissible set $B\subset\mathcal A$ with
\[
  \PR(X\in\mathcal A)-\PR(X\in B)\ \le\ \sum_{u\in I_{g,\mathcal A}}m_u .
\]
\end{lemma}

\begin{proof}
By \eqref{eq:wdef}, $w_u\ge\prod_ip_{a_i}^{m_i}\ge m_u^{\sum_im_i}=m_u^h$. For
each $u\in I_{g,\mathcal A}$ pick an element of $u$ of minimal probability and
let $D$ be the set of picked points, so
$\PR(X\in D)\le\sum_{I_{g,\mathcal A}}m_u$; put $B=\mathcal A\setminus D$. Every
$u\in\Uc_h(B)$ has all its elements in $\mathcal A$ and avoids $D$, so
$u\notin I_g$, that is, $u$ is among the $g$ heaviest elements of its $f$-fibre.
Each fibre of $f$ therefore meets $\Uc_h(B)$ in at most $g$ multisets, so $B$ is
$(f,g)$-admissible.
\end{proof}

\begin{lemma}[R\'enyi entropy budget estimates]\label{lem:tail}
Let $A_\tau=\{a:p_a\ge\tau\}$, so $\lvert A_\tau\rvert\le\lfloor1/\tau\rfloor$.
\begin{enumerate}
\item If $0<\beta<1$ and $H_\beta(X)\le D$ then $P_\beta\le e^{(1-\beta)D}$,
\IfTwoCol{%
\[
  \begin{aligned}
  &\PR(X\notin A_\tau)\le e^{(1-\beta)D}\tau^{1-\beta},\\
  &P_\alpha\le P_\beta\ \text{ for every }\alpha\ge\beta .
  \end{aligned}
\]}{%
\[
  \PR(X\notin A_\tau)\le e^{(1-\beta)D}\tau^{1-\beta},
  \quad\text{and}\quad
  P_\alpha\le P_\beta\ \text{ for every }\alpha\ge\beta .
\]}
\item If $H(X)\le D<\infty$ then $\PR(X\notin A_\tau)\le D/\log(1/\tau)$, and
$P_\alpha\le1$ for every $\alpha\ge1$.
\item If $\beta>1$ and $H_\beta(X)\le D$ then $P_\beta\ge e^{-(\beta-1)D}$ and
$\sum_{a\notin A_\tau}p_a^\beta\le\tau^{\beta-1}$.
\end{enumerate}
\end{lemma}

\begin{proof}
(1) $H_\beta(X)\le D$ means $\log P_\beta\le(1-\beta)D$; and $p_a<\tau$ gives
$p_a\le p_a^\beta\tau^{1-\beta}$, so summing yields the tail bound. For the last
assertion, $p_a\le1$ and $\alpha\ge\beta$ give $p_a^\alpha\le p_a^\beta$
termwise.
(2) For $a\notin A_\tau$ we have $\log(1/p_a)\ge\log(1/\tau)$, so
$D\ge H(X)\ge\sum_{a\notin A_\tau}p_a\log(1/p_a)
\ge\log(1/\tau)\PR(X\notin A_\tau)$; and $p_a^\alpha\le p_a$ for $\alpha\ge1$.
(3) Now $1/(1-\beta)<0$, so $H_\beta(X)\le D$ reads
$\log P_\beta\ge(1-\beta)D$. And for $a\notin A_\tau$,
$p_a^\beta=p_a\,p_a^{\beta-1}\le p_a\tau^{\beta-1}$, so summing gives the second
bound.
\end{proof}

\begin{lemma}[Excess-fibre power sum]\label{lem:mass}
Let $\alpha\ne1$. Then
\IfTwoCol{%
\begin{equation}\label{eq:LDelta}
\begin{split}
  \Lc_{\alpha,g}(f)
  &=M_\alpha(U_h)\,\bigl\lvert1-e^{-(1-\alpha)\Delta^{[g]}_{\alpha,f}(X)}
   \bigr\rvert,\\
  T_{\alpha,g}(f)&\ \le\ \frac{\Lc_{\alpha,g}(f)}{d_\alpha} .
\end{split}
\end{equation}}{%
\begin{equation}\label{eq:LDelta}
  \Lc_{\alpha,g}(f)
  =M_\alpha(U_h)\,\bigl\lvert1-e^{-(1-\alpha)\Delta^{[g]}_{\alpha,f}(X)}
   \bigr\rvert
  \qquad\text{and}\qquad
  T_{\alpha,g}(f)\ \le\ \frac{\Lc_{\alpha,g}(f)}{d_\alpha} .
\end{equation}}
Consequently, let $\Delta^{[g]}_{\alpha,f}(X)\le C$ and put
\begin{equation}\label{eq:Kall}
  K_{\alpha,\beta,D,h}(C):=\begin{cases}
   \dfrac{\bigl(1-e^{-(1-\alpha)C}\bigr)e^{h(1-\beta)D}}{d_\alpha},
     & \alpha<1,\\[3mm]
   \dfrac{C}{2\log2}, & \alpha=1,\\[3mm]
   \dfrac{(h!)^{\alpha-1}\bigl(e^{(\alpha-1)C}-1\bigr)}{d_\alpha},
     & \alpha>1 .
  \end{cases}
\end{equation}
Then $K_{\alpha,\beta,D,h}(C)=O(C)$ and
\begin{equation}\label{eq:K}
  \sum_{u\in I_g}w_u^{\,\alpha}\ \le\ K_{\alpha,\beta,D,h}(C) .
\end{equation}
For $\alpha<1$ this requires $H_\beta(X)\le D$ for some $\beta\le\alpha$. For
$\alpha=1$ it requires only $H(X)<\infty$, which is in any case part of the
standing finiteness hypothesis of Proposition~\ref{prop:basic} at that order, and
for $\alpha>1$ it requires nothing about $X$ at all. Thus for $\alpha\ge1$ the
bound is budget-free and $K_{\alpha,\beta,D,h}$ depends on neither $\beta$ nor
$D$; we then abbreviate it $K_{\alpha,h}$.
\end{lemma}

\begin{proof}
For each $\chi$, Proposition~\ref{prop:basic}(3) applied to $\pi=(f,\chi)$ gives
\IfTwoCol{%
\[
  \begin{aligned}
  &\bigl\lvert M_\alpha(U_h)-M_\alpha\bigl((f,\chi)(U_h)\bigr)\bigr\rvert\\
  &\qquad=M_\alpha(U_h)\bigl\lvert
    1-e^{-(1-\alpha)\Delta_{\alpha,(f,\chi)}(X)}\bigr\rvert ,
  \end{aligned}
\]}{%
\[
  \bigl\lvert M_\alpha(U_h)-M_\alpha\bigl((f,\chi)(U_h)\bigr)\bigr\rvert
  =M_\alpha(U_h)\bigl\lvert1-e^{-(1-\alpha)\Delta_{\alpha,(f,\chi)}(X)}
   \bigr\rvert ,
\]}
which is an increasing function of $\Delta_{\alpha,(f,\chi)}(X)$ for
$\alpha<1$ and for $\alpha>1$ alike; taking the infimum over $\chi$ on both
sides gives the first half of \eqref{eq:LDelta}. The second half is
Theorem~\ref{thm:coarse} with $\pi=f$, $\lambda=\chi$, followed by the infimum
over $\chi$, the right-hand side $d_\alpha T_{\alpha,g}(f)$ being independent of
$\chi$.

For $\alpha<1$ and $\Delta^{[g]}_{\alpha,f}\le C$ we get
$\Lc_{\alpha,g}(f)\le M_\alpha(U_h)(1-e^{-(1-\alpha)C})$, and, since then
$\beta\le\alpha<1$,
$M_\alpha(U_h)\le P_\alpha^h\le P_\beta^h\le e^{h(1-\beta)D}$ by
Proposition~\ref{prop:MU} and Lemma~\ref{lem:tail}(1); now use $1-e^{-x}\le x$
to see that the first line of \eqref{eq:Kall} is $O(C)$. For $\alpha>1$ use
$\Lc_{\alpha,g}(f)\le M_\alpha(U_h)(e^{(\alpha-1)C}-1)$ together with
$M_\alpha(U_h)\le(h!)^{\alpha-1}P_\alpha^h\le(h!)^{\alpha-1}$, valid since
$P_\alpha\le1$ for $\alpha\ge1$; no budget is needed in this case. The case
$\alpha=1$ is Lemma~\ref{lem:shannonmass}, where the exponent on $w_u$ is $1$.
\end{proof}

\begin{lemma}[Shannon excess-fibre sum]\label{lem:shannonmass}
Let $H(X)<\infty$ and $\Delta^{[g]}_{1,f}(X)\le C$. Then
\[
  \sum_{u\in I_g}w_u\ \le\ \frac{C}{2\log2} .
\]
\end{lemma}

\begin{proof}
Fix $\chi$ and write $\sigma(u)$ for the total weight of the class of $u$, that
is of $\{u':f(u')=f(u),\ \chi(u')=\chi(u)\}$. Since
$H\bigl((f,\chi)(U_h)\bigr)\le H(U_h)\le hH(X)<\infty$,
\begin{equation}\label{eq:shannonid}
  \Delta_{1,(f,\chi)}(X)=H\bigl(U_h\mid(f,\chi)(U_h)\bigr)
  =\E\Bigl[\log\frac{\sigma(U_h)}{w_{U_h}}\Bigr],
\end{equation}
which is the left side of \eqref{eq:collisionshannon} for $\pi=f$ and
$\lambda=\chi$. Theorem~\ref{thm:coarse}(2) therefore gives
$\Delta_{1,(f,\chi)}(X)\ge(2\log2)T_{1,g}(f)=(2\log2)\sum_{u\in I_g}w_u$ for every
$\chi$, and we may take the infimum.
\end{proof}

\section{The removal principle}\label{sec:upper}

\begin{theorem}[Weighted collision-removal principle]\label{thm:general}
Let $h\ge2$, $g\ge1$, let $f$ be any map on $\Uc_h$, let $\alpha,\beta>0$, and
suppose $\Delta^{[g]}_{\alpha,f}(X)\le C$, where $\Delta^{[g]}_{\alpha,f}$ is
the $g$-list deficit \eqref{eq:listdeficit} of $f$; parts (1)--(3) assume
$\beta\le\alpha$, while part (4) assumes only $\alpha\ge1$ and $\beta>1$. Write
$K_{\alpha,\beta,D,h}(C)$ for the quantity \eqref{eq:Kall}, which is
$O_{\alpha,\beta,D,h}(C)$.
\begin{enumerate}
\item If $h\alpha\le1$ then
$\delta_{f,g}(X)\le\bigl(\Lc_{\alpha,g}(f)/d_\alpha\bigr)^{1/(h\alpha)}$, with no
budget at all; if moreover $H_\beta(X)\le D$ then
$\delta_{f,g}(X)=O_{\alpha,\beta,D,h}\bigl(C^{1/(h\alpha)}\bigr)$.
\item If $h\alpha>1$, $\beta<1$ and $H_\beta(X)\le D$, then for every
$\tau\in(0,1)$
\[
  \delta_{f,g}(X)\le e^{(1-\beta)D}\tau^{1-\beta}
   +K_{\alpha,\beta,D,h}(C)\,\tau^{-(h\alpha-1)},
\]
and optimizing $\tau$ gives
$\delta_{f,g}(X)=O_{\alpha,\beta,D,h}
\bigl(C^{(1-\beta)/(h\alpha-\beta)}\bigr)$.
\item If $h\alpha>1$, $\beta=1$ and $H(X)\le D<\infty$, then for every
$\tau\in(0,1)$
\[
  \delta_{f,g}(X)\le\frac{D}{\log(1/\tau)}
   +K_{\alpha,1,D,h}(C)\,\tau^{-(h\alpha-1)},
\]
and hence $\delta_{f,g}(X)\le(1+o(1))(h\alpha-1)D/\log(1/C)$ as $C\to0$.
\item If $\alpha\ge1$, $\beta>1$ and $H_\beta(X)\le D$, where $\alpha\ge\beta$
is \emph{not} assumed, then
\IfTwoCol{%
\[
  \begin{aligned}
  \delta_{f,g}(X)\ \le\ 1&-e^{-(\beta-1)D/\beta}\\
   &+O_{\alpha,\beta,D,h}\bigl(C^{\min\{1,\,(\beta-1)/(h\alpha-1)\}}\bigr).
  \end{aligned}
\]}{%
\[
  \delta_{f,g}(X)\ \le\ 1-e^{-(\beta-1)D/\beta}
   +O_{\alpha,\beta,D,h}\bigl(C^{\min\{1,\,(\beta-1)/(h\alpha-1)\}}\bigr).
\]}
\end{enumerate}
\end{theorem}

A small R\'enyi coarsening loss therefore forces every fibre of $f$ to shrink to
at most $g$ multisets after deleting a controlled amount of probability mass. The
mechanism is uniform: delete inside the set of atoms of weight at least $\tau$,
pay the tail, and trade the two errors against each other. What changes with
$\beta$ is only what the budget buys. For $\beta<1$ it bounds the tail mass by
$\tau^{1-\beta}$; at $\beta=1$ only by $1/\log(1/\tau)$, which is why the rate
degenerates to logarithmic; and for $\beta>1$ it bounds no tail mass at all, but
it does bound the $\beta$-th power sum from below, and that is enough to keep a
fixed fraction of the weight, never all of it. Additive structure enters only
through the constructions of Section~\ref{sec:lower}, which show that ordinary
integer addition already realizes every exponent in \eqref{eq:theta}.

\begin{proof}[Proof of Theorem~\ref{thm:rawmoment}]
Apply Lemma~\ref{lem:deletion} with $\mathcal A=A$, so that
$\delta_{f,g}(X)\le\sum_{u\in I_g}m_u$. Since $h\alpha\le1$ the function
$t\mapsto t^{h\alpha}$ is subadditive, and $m_u^{h\alpha}\le w_u^\alpha$ because
$w_u\ge m_u^h$; hence
\[
  \Bigl(\sum_{u\in I_g}m_u\Bigr)^{h\alpha}\le\sum_{u\in I_g}m_u^{h\alpha}
  \le\sum_{u\in I_g}w_u^\alpha=T_{\alpha,g}(f)
  \le\frac{\Lc_{\alpha,g}(f)}{d_\alpha} ,
\]
the last step by the second half of \eqref{eq:LDelta}.
\end{proof}

\begin{proof}[Proof of Corollary~\ref{cor:multimap}]
For each $i$, Lemma~\ref{lem:deletion} with $\mathcal A=A$ and $f=f_i$ gives an
$(f_i,g)$-admissible set $B_i$ with
$\PR(X\notin B_i)\le\sum_{u\in I_g(f_i)}m_u$, and the proof of
Theorem~\ref{thm:rawmoment} bounds this sum by
$(\Lc_{\alpha,g}(f_i)/d_\alpha)^{1/(h\alpha)}$. A subset of an
$(f_i,g)$-admissible set is $(f_i,g)$-admissible, since its multisets form a
subcollection, so $B=\bigcap_iB_i$ is admissible for every $f_i$ and the union
bound gives the estimate. For the particular case take $h=2$, $g=1$,
$\alpha=\frac12$, $f_1=s$ and $f_2$ the product, so that $h\alpha=1$ and
$d_{1/2}=2-\sqrt2$.
\end{proof}

\begin{proof}[Proof of Theorem~\ref{thm:general}]
Throughout write $K=K_{\alpha,\beta,D,h}(C)$ for the quantity
\eqref{eq:Kall}. Thus $K=O(C)$, and \eqref{eq:K} bounds
$\sum_{u\in I_g}w_u^\alpha$ by $K$. Note also that for $u\in I_{g,A_\tau}$
every element of $u$ has probability at least $\tau$, so $m_u\ge\tau$; as
$h\alpha>1$ in parts (2)--(4), this gives
\begin{equation}\label{eq:mubound}
  m_u=m_u^{h\alpha}\,m_u^{1-h\alpha}\ \le\ w_u^{\,\alpha}\,\tau^{-(h\alpha-1)}
\end{equation}
by $w_u\ge m_u^h$.

(1) The first assertion is Theorem~\ref{thm:rawmoment}. For the second, the
chain displayed in the proof of that theorem gives
$\delta_{f,g}(X)^{h\alpha}\le T_{\alpha,g}(f)$ directly, without passing through
$\Lc_{\alpha,g}(f)$, and $T_{\alpha,g}(f)=\sum_{u\in I_g}w_u^\alpha\le K$ by
\eqref{eq:K}; hence $\delta_{f,g}(X)\le K^{1/(h\alpha)}$, which is
$O_{\alpha,\beta,D,h}(C^{1/(h\alpha)})$.

(2) Fix $\tau\in(0,1)$ and apply Lemma~\ref{lem:deletion} with
$\mathcal A=A_\tau$, which is finite. Then
$\delta_{f,g}(X)\le\PR(X\notin A_\tau)+\sum_{u\in I_{g,A_\tau}}m_u$, the first
term being at most $e^{(1-\beta)D}\tau^{1-\beta}$ by Lemma~\ref{lem:tail}(1) and
the second at most $K\tau^{-(h\alpha-1)}$ by \eqref{eq:mubound} and
\eqref{eq:K}. This is the displayed bound. Writing $A_1=e^{(1-\beta)D}$ and
choosing $\tau=(K/A_1)^{1/(h\alpha-\beta)}$, which lies in $(0,1)$ once
$K<A_1$, makes the two terms equal, because
\begin{equation}\label{eq:balance}
  (1-\beta)+(h\alpha-1)=h\alpha-\beta>0 ;
\end{equation}
each then equals $A_1(K/A_1)^{(1-\beta)/(h\alpha-\beta)}$, which is
$O_{\alpha,\beta,D,h}\bigl(C^{(1-\beta)/(h\alpha-\beta)}\bigr)$.

(3) The same argument with the Shannon tail bound Lemma~\ref{lem:tail}(2) in
place of the R\'enyi one gives the displayed bound. Writing $L=\log(1/C)$ and
taking $\tau=(CL^2)^{1/(h\alpha-1)}$, which lies in $(0,1)$ once $CL^2<1$, we
get $\tau^{h\alpha-1}=CL^2$ and $\log(1/\tau)=(L-2\log L)/(h\alpha-1)$, so,
since $K=O(C)$,
\begin{equation}\label{eq:critrate}
  \delta_{f,g}(X)\le\frac{(h\alpha-1)D}{L-2\log L}+O\Bigl(\frac{1}{L^2}\Bigr)
  =(1+o(1))\frac{(h\alpha-1)D}{L}.
\end{equation}

(4) Here $\alpha\ge1$ and $\beta>1$, and $\alpha\ge\beta$ is not assumed; note
that $K=K_{\alpha,h}$ needs no budget, by the last sentence of
Lemma~\ref{lem:mass}. Lemma~\ref{lem:tail}(1) is unavailable: the
budget bounds no tail mass. It does, however, bound a power sum from below. Put
$R=e^{-(\beta-1)D}$, so that $P_\beta\ge R$ by Lemma~\ref{lem:tail}(3). We first
convert that lower bound on $P_\beta$ into a lower bound on the mass retained by
$B$. Fix
$\tau\in(0,1)$ and let $B\subset A_\tau$ be the $(f,g)$-admissible set produced
by Lemma~\ref{lem:deletion} with $\mathcal A=A_\tau$, obtained by deleting one
element of minimal probability from each $u\in I_{g,A_\tau}$. Then
\IfTwoCol{%
\[
  \begin{aligned}
  \sum_{a\in B}p_a^{\,\beta}
  &\ \ge\ P_\beta-\sum_{a\notin A_\tau}p_a^{\,\beta}
   -\sum_{u\in I_{g,A_\tau}}m_u^{\,\beta}\\
  &\ \ge\ R-\tau^{\beta-1}-\sum_{u\in I_{g,A_\tau}}m_u^{\,\beta}
  \end{aligned}
\]}{%
\[
  \sum_{a\in B}p_a^{\,\beta}\ \ge\ P_\beta-\sum_{a\notin A_\tau}p_a^{\,\beta}
   -\sum_{u\in I_{g,A_\tau}}m_u^{\,\beta}
  \ \ge\ R-\tau^{\beta-1}-\sum_{u\in I_{g,A_\tau}}m_u^{\,\beta}
\]}
by Lemma~\ref{lem:tail}(3). For $u\in I_{g,A_\tau}$ we have $m_u\ge\tau$ and $m_u\le1$, so
$m_u^{\,\beta}=m_u^{h\alpha}m_u^{\,\beta-h\alpha}
\le w_u^{\,\alpha}\tau^{-(h\alpha-\beta)_+}$, where $x_+=\max\{x,0\}$: if
$\beta\le h\alpha$ use $m_u\ge\tau$ and if $\beta>h\alpha$ use $m_u\le1$.
Hence that last sum is at most $K\tau^{-(h\alpha-\beta)_+}$. Since $\beta>1$ and $p_a\le\sum_{a'\in B}p_{a'}$ for
$a\in B$, we have
$\sum_{a\in B}p_a^{\,\beta}\le\bigl(\sum_{a\in B}p_a\bigr)^{\beta}$, so
\[
  \PR(X\in B)\ \ge\ \Bigl(R-\tau^{\beta-1}-K\tau^{-(h\alpha-\beta)_+}
   \Bigr)^{1/\beta} .
\]
\emph{Case 1: $\beta\le h\alpha$.} Here $(h\alpha-\beta)_+=h\alpha-\beta$; take
$\tau=C^{1/(h\alpha-1)}$: then $\tau^{\beta-1}=C^{(\beta-1)/(h\alpha-1)}$ and,
because $1-\frac{h\alpha-\beta}{h\alpha-1}=\frac{\beta-1}{h\alpha-1}$ and
$K=O(C)$, also
$K\tau^{-(h\alpha-\beta)}=O\bigl(C^{(\beta-1)/(h\alpha-1)}\bigr)$.

\emph{Case 2: $\beta>h\alpha$.} Here $(h\alpha-\beta)_+=0$; take instead
$\tau=C^{1/(\beta-1)}$, and then $\tau^{\beta-1}$ and $K$ are both $O(C)$.

In either case the two errors are
$O\bigl(C^{\min\{1,(\beta-1)/(h\alpha-1)\}}\bigr)$. As $x\mapsto x^{1/\beta}$ is
Lipschitz on $[R/2,1]$, for all small $C$ this gives
$\PR(X\in B)\ge R^{1/\beta}-O\bigl(C^{\min\{1,(\beta-1)/(h\alpha-1)\}}\bigr)$,
and $R^{1/\beta}=e^{-(\beta-1)D/\beta}$.
\end{proof}

\begin{remark}
The exponent $h$ in $w_u\ge m_u^h$ is the only place where $h$ enters
Theorems~\ref{thm:rawmoment} and~\ref{thm:general}, and always through $h\alpha$:
subadditivity of $t\mapsto t^{h\alpha}$ in part (1) needs $h\alpha\le1$, and
\eqref{eq:mubound} needs $h\alpha>1$. The budget order enters only through
Lemma~\ref{lem:tail}, whose three cases are parts 1--3 of
Theorem~\ref{thm:phase}. The fourth regime lies outside the reach of
Theorem~\ref{thm:general} altogether, because there the budget fails to control
$P_\alpha$ rather than failing to control a tail. Finally, for fixed $\alpha$
with $h\alpha>1$ the exponent $(1-\beta)/(h\alpha-\beta)$ of part (2) tends to
$0$ as $\beta\uparrow1$, and for each fixed $C$ the optimized bound of part (2)
tends to $1$, which is trivial; so the polynomial law does not pass to the
logarithmic law of part (3) by a limit in $\beta$ alone. The two are joined on
the scale $(1-\beta)\log(1/C)\asymp1$, by Theorem~\ref{thm:crossover}.
\label{rem:betato1}
\end{remark}

\begin{corollary}[Explicit upper bound in Theorem~\ref{thm:main}]\label{cor:lgk}
For $C,D\ge0$ put
\[
  \widehat f(C,D)=\min\Bigl\{1,\ \inf_{0<\tau<1}
  \Bigl(\frac{D}{\log(1/\tau)}+\frac{C}{2\tau\log2}\Bigr)\Bigr\}\in[0,1].
\]
Then $\widehat f(C,D)=(1+o(1))D/\log(1/C)$ as $C\to0$, and every $X$ with
$H(X)\le D$ and $H(U_2\mid S_2)=\Delta_{1,2}(X)\le C$ admits a Sidon set
$B\subset A$ with $\PR(X\in B)\ge1-\widehat f(C,D)$. This is the function asked
for in \cite[after Example~5.4]{LGK}, and the upper bound of
Theorem~\ref{thm:main}.
\end{corollary}

\begin{proof}
Theorem~\ref{thm:general}(3) with $h=2$, $\alpha=\beta=1$, $g=1$ and $f=s$
gives, for each
$\tau\in(0,1)$, a Sidon set $B_\tau$ with $\PR(X\in B_\tau)\ge1-\gamma(\tau)$,
where $\gamma(\tau)=D/\log(1/\tau)+C/(2\tau\log2)$, and \eqref{eq:critrate} gives
$\widehat f(C,D)\le(1+o(1))D/\log(1/C)$; the reverse inequality holds because
$\widehat f\ge\Phi_{1,1,D,2}$ and Theorem~\ref{thm:phase}(2) gives
$\Phi_{1,1,D,2}(C)\log(1/C)\to D$. If $C,D>0$ then $\gamma$ is continuous on $(0,1)$ and
tends to $\infty$ at both ends, so it attains its infimum at some $\tau_*$, and
$B=B_{\tau_*}$ works; if $\inf\gamma>1$ the claim is trivial. If $C=0$ then
$\Delta_{1,2}(X)=0$, so $A$ is Sidon by Proposition~\ref{prop:listzero} and
$B=A$ works; if $D=0$ then $A$ is a singleton.
\end{proof}

\section{Blocks from a \texorpdfstring{$B_{h-1}$}{B\_\{h-1\}} set}
\label{sec:blocks}

All the sharpness constructions come from one template. It is built from a
finite block $F$ that is $B_{h-1}$, so that no collision straddles two copies of
$F$, and is not $B_h[g]$, so that each copy forces a deletion. We first record
that such blocks exist for every $h$ and $g$, and that they can be taken with
vanishing $B_h[g]$-density: for every $h\ge2$ and fixed $g\ge1$ there are finite
integer sets $F$ that are $B_{h-1}$ and yet satisfy
\begin{equation}\label{eq:hierarchy}
  \frac{b_{h,g}(F)}{\lvert F\rvert}
  \;=\;O_{h,g}\bigl(\lvert F\rvert^{-1/h}\bigr)\ \longrightarrow\ 0 ,
\end{equation}
a separation between consecutive levels of the $B_h$ hierarchy.

\begin{lemma}[{thin $B_h[g]$ subsets of a thick $B_{h-1}$ set}]\label{lem:Fbound}
Let $h\ge2$, $g\ge1$ and let $F\subset\{1,\dots,M\}$ be finite. Then every
$B_h[g]$ subset $S\subseteq F$ satisfies
\IfTwoCol{%
\begin{equation}\label{eq:Fbound}
  \binom{\lvert S\rvert+h-1}{h}\;\le\;g\,hM,
  \quad\text{hence}\quad
  \lvert S\rvert\;\le\;\bigl(g\,h\cdot h!\,M\bigr)^{1/h}.
\end{equation}}{%
\begin{equation}\label{eq:Fbound}
  \binom{\lvert S\rvert+h-1}{h}\;\le\;g\,hM,
  \qquad\text{hence}\qquad
  \lvert S\rvert\;\le\;\bigl(g\,h\cdot h!\,M\bigr)^{1/h}.
\end{equation}}
Consequently, if $F_1,F_2,\dots$ is a sequence of finite subsets of $\Z_{>0}$
with $F_i\subset\{1,\dots,M_i\}$ and $M_i=o\bigl(\lvert F_i\rvert^{h}\bigr)$,
then for every fixed $g$ the largest $B_h[g]$ subset of $F_i$ has size
$o(\lvert F_i\rvert)$.
\end{lemma}

\begin{proof}
If $S$ is $B_h[g]$ then the map $s:\Uc_h(S)\to\{h,\dots,hM\}$ is at most
$g$-to-one, and $\lvert\Uc_h(S)\rvert=\binom{\lvert S\rvert+h-1}{h}$ while the
target has at most $hM$ elements. This is the first inequality, and the second
follows from $\binom{n+h-1}{h}\ge n^h/h!$. For the last statement,
$\lvert S\rvert/\lvert F_i\rvert\le(g\,h\cdot h!)^{1/h}
M_i^{1/h}/\lvert F_i\rvert\to0$.
\end{proof}

\begin{lemma}[Bose--Chowla]\label{lem:BC}
For every $h\ge2$ there is a sequence of finite sets
$F^{(h)}_1,F^{(h)}_2,\dots\subset\Z_{>0}$ such that each $F^{(h)}_i$ is a
$B_{h-1}$ set, $\lvert F^{(h)}_i\rvert\to\infty$, and
$F^{(h)}_i\subset\{1,\dots,M_i\}$ with
$M_i\le\lvert F^{(h)}_i\rvert^{\,h-1}$. Consequently, by
Lemma~\ref{lem:Fbound}, for every fixed $g\ge1$
\[
  \frac{b_{h,g}\bigl(F^{(h)}_i\bigr)}{\lvert F^{(h)}_i\rvert}
  \;=\;O_{h,g}\bigl(\lvert F^{(h)}_i\rvert^{-1/h}\bigr)
  \;\longrightarrow\;0 ,
\]
which is \eqref{eq:hierarchy}. In particular, for every $h\ge2$ and $g\ge1$
there exist finite $F\subset\Z_{>0}$ that are $B_{h-1}$ but not $B_h[g]$.
\end{lemma}

\begin{proof}
For $h=2$ the condition ``$B_1$'' is vacuous and we may take
$F^{(2)}_i=\{1,\dots,i\}$, so that $M_i=i=\lvert F^{(2)}_i\rvert$. For $h\ge3$
put $r=h-1\ge2$ and let $q$ run over the prime powers. The Bose--Chowla
construction \cite{BoseChowla} produces a $B_r$ set of size $q$ in
$\Z/(q^{r}-1)\Z$. Choosing representatives in $\{1,\dots,q^{r}-1\}$ gives a
$B_r$ set in $\Z$: any equality of integer $r$-fold sums implies the
corresponding congruence modulo $q^{r}-1$, and the multisets therefore agree.
See \cite{Cilleruelo} for bounds on finite $B_h[g]$ sequences and
\cite{OBryantThick} for thick $B_h$ constructions. Taking $F^{(h)}_q$ to be that
set gives $\lvert F^{(h)}_q\rvert=q$ and $M\le q^{r}-1<q^{h-1}$, as required.
Since $h-1<h$ we have $M=o(\lvert F\rvert^{h})$, so Lemma~\ref{lem:Fbound}
applies and gives the displayed rate. The last assertion follows because
$b_{h,g}(F)/\lvert F\rvert<1$ for large $i$.
\end{proof}

\begin{remark}\label{rem:Zversion}
For $h=2$ Lemma~\ref{lem:BC} uses only $F=\{1,\dots,i\}$, and
Lemma~\ref{lem:Fbound} reduces to the classical bound
$\binom{b_{2,g}(F)+1}{2}\le2gi$, i.e.\ $b_{2,g}(F)=O(\sqrt{gi})$. For $h\ge3$
the input is the Bose--Chowla $B_{h-1}$ set, whose thickness
$M<\lvert F\rvert^{h-1}$ is what makes \eqref{eq:Fbound} nontrivial. A
finite-field analogue can be obtained from the moment curve
$\{(t,\dots,t^{h-1}):t\in\mathbb F_p\}$, which is $B_{h-1}$ by Newton's
identities when $p>h-1$, but we do not use it here.
\end{remark}

\subsection{The block template and its invariants}

\begin{example}[separated blocks]\label{ex:blocks}
Let $h\ge2$ and $g\ge1$, let $F\subset\{1,\dots,M\}$ be a finite set that is
$B_{h-1}$ and not $B_h[g]$. Let $B>hM$ be an integer, let $N\ge1$, and put
\IfTwoCol{%
\[
  \begin{aligned}
  A_{F,N}&:=\bigl\{aB^{k}+B^{N+k}:a\in F,\ 0\le k<N\bigr\}\subset\Z,\\
  y&:=B^{2N+1},
  \end{aligned}
\]}{%
\[
  A_{F,N}:=\bigl\{aB^{k}+B^{N+k}\;:\;a\in F,\ 0\le k<N\bigr\}\subset\Z,
  \qquad y:=B^{2N+1},
\]}
so $\lvert A_{F,N}\rvert=\lvert F\rvert N$ and $y\notin A_{F,N}$. For $q\in(0,1)$ let
$Y=Y_{F,N}^{(q)}$ have $\PR(Y=y)=1-q$ and $\PR(Y=a)=r:=q/(\lvert F\rvert N)$ for every
$a\in A_{F,N}$. We call the $N$ sets $\{aB^k+B^{N+k}:a\in F\}$ the
\emph{blocks}, and $y$ the \emph{heavy atom}.
\end{example}

The invariants of the template are governed by two constants attached to $F$
alone. For $\alpha\ne1$ let $\Rc^{[g]}_{\alpha,h}(F)$ be as in \eqref{eq:Rg},
and for the Shannon order put
\IfTwoCol{%
\begin{equation}\label{eq:LambdaF}
  \begin{aligned}
    \Lambda^{[g]}_h(F)&:=\inf_{\chi:\Uc_h(F)\to\{1,\dots,g\}}\
     \sum_{(z,c)}\ \sum_{\substack{u\in\Uc_h(F)\\ s(u)=z,\ \chi(u)=c}}
      \nu_u\log\frac{\nu_{z,c}}{\nu_u},\\
    \nu_{z,c}&=\!\!\sum_{\substack{u\,:\,s(u)=z\\ \chi(u)=c}}\!\!\nu_u .
  \end{aligned}
\end{equation}}{%
\begin{equation}\label{eq:LambdaF}
  \Lambda^{[g]}_h(F):=\inf_{\chi:\Uc_h(F)\to\{1,\dots,g\}}\
   \sum_{(z,c)}\ \sum_{\substack{u\in\Uc_h(F)\\ s(u)=z,\ \chi(u)=c}}
    \nu_u\log\frac{\nu_{z,c}}{\nu_u},
  \qquad \nu_{z,c}=\!\!\sum_{\substack{u\,:\,s(u)=z\\ \chi(u)=c}}\!\!\nu_u .
\end{equation}}

\begin{lemma}\label{lem:blockconst}
For every finite $F$, every $h\ge2$, $g\ge1$ and $\alpha>0$ with $\alpha\ne1$,
\[
  \begin{aligned}
  \Rc^{[g]}_{\alpha,h}(F)&\ \ge\ d_\alpha\sum_z\sum_{j>g}
    \nu_{z,j}^{\,\alpha},\\
  \Lambda^{[g]}_h(F)&\ \ge\ (2\log2)\sum_z\sum_{j>g}\nu_{z,j},
  \end{aligned}
\]
where $\nu_{z,1}\ge\nu_{z,2}\ge\cdots$ are the multiplicities $\nu_u$ of the
$u\in\Uc_h(F)$ with $s(u)=z$. In particular both quantities are finite and
nonnegative, and both are strictly positive if and only if $F$ is not a
$B_h[g]$ set.
\end{lemma}

\begin{proof}
Normalize the weights $\nu_u$ to a probability vector on $\Uc_h(F)$ and apply
Theorem~\ref{thm:coarse}(1), respectively (2), with $\pi=s$; both right-hand sides are independent of $\chi$, so
the infima may be taken. The sums $\sum_{j>g}$ are nonempty for some $z$ exactly
when some fibre of $s$ on $\Uc_h(F)$ has more than $g$ elements, that is exactly
when $F$ is not $B_h[g]$. Conversely, if $F$ is $B_h[g]$ then a $\chi$ that is
injective on each fibre makes both expressions vanish.
\end{proof}

\begin{lemma}[exact invariants of the template]\label{lem:blocks}
For $Y=Y_{F,N}^{(q)}$ as in Example~\ref{ex:blocks}:
\begin{enumerate}
\item \emph{Fibre structure.} Two distinct multisets in $\Uc_h$ have equal sums if and only if, for a
single $k$, both consist of $h$ points of block $k$ whose $F$-parameter multisets
have equal sums; in particular no such pair involves $y$. The nonsingleton fibres
of $s$ on $\Uc_h$ are therefore $N$ disjoint copies, one per block, of the
nonsingleton fibres of $s$ on $\Uc_h(F)$. Here the multiset of block $k$ with
parameter multiset $T$ carries weight $\nu_T\,r^{h}$;
\item \emph{Deletion cost.} A subset of $A_{F,N}\cup\{y\}$ is $B_h[g]$ if and only if, for each
$0\le k<N$, the set of parameters it uses in block $k$ is a $B_h[g]$ subset of
$F$; hence
\[
  \delta_{h,g}(Y)=q\Bigl(1-\frac{b_{h,g}(F)}{\lvert F\rvert}\Bigr)\;>\;0 ;
\]
\item \emph{Entropy budget.} $H(Y)=\eta(q)+q\log(\lvert F\rvert N)$ and
$P_\alpha(Y)=(1-q)^\alpha+(\lvert F\rvert N)^{1-\alpha}q^\alpha$, both exactly;
\item \emph{R\'enyi deficit.} For $\alpha\ne1$, exactly
$\ \Lc_{\alpha,g}(s)=N\,r^{h\alpha}\,\Rc^{[g]}_{\alpha,h}(F)$, and hence
\IfTwoCol{%
\[
  \begin{aligned}
  \Delta^{[g]}_{\alpha,h}(Y)&=\frac{1}{1-\alpha}
   \log\frac{M_\alpha(U_h)}{M_\alpha(U_h)-\varsigma N r^{h\alpha}
    \Rc^{[g]}_{\alpha,h}(F)},\\
   \varsigma&=\operatorname{sign}(1-\alpha);
  \end{aligned}
\]}{%
\[
  \Delta^{[g]}_{\alpha,h}(Y)=\frac{1}{1-\alpha}
   \log\frac{M_\alpha(U_h)}{M_\alpha(U_h)-\varsigma N r^{h\alpha}
    \Rc^{[g]}_{\alpha,h}(F)},\qquad
   \varsigma=\operatorname{sign}(1-\alpha);
\]}
\item \emph{Shannon deficit.} $\Delta^{[g]}_{1,h}(Y)=N\Lambda^{[g]}_h(F)\,r^{h}
  =\Lambda^{[g]}_h(F)\,q^{h}/(\lvert F\rvert^{h}N^{h-1})$ exactly.
\end{enumerate}
\end{lemma}

\begin{proof}
(1) Let $u\in\Uc_h$ contain $n_y$ copies of $y$ and, for each $k$, a multiset
$T_k$ of $n_k=\lvert T_k\rvert$ parameters from $F$, so $n_y+\sum_kn_k=h$. Each
element of $F$ is at most $M$ and at most $h$ summands occur, so the coefficient
appearing in base-$B$ position $k$ is at most $hM<B$, and those in positions
$N+k$ and $2N+1$ are at most $h<B$. Hence no carry occurs, and the digits of
$s(u)$ in positions $k$, $N+k$, $2N+1$ return $\sum_{a\in T_k}a$, $n_k$ and
$n_y$ respectively. Two multisets in one fibre of $s$ therefore agree in $n_y$
and in every pair $(n_k,\sum_{a\in T_k}a)$. If they are distinct, then
$T_k\neq T_k'$ for some $k$, with $\lvert T_k\rvert=\lvert T_k'\rvert=n_k$ and
equal sums. If $n_k\le h-1$, pad both by $h-1-n_k$ copies of a fixed element of
$F$; the padded multisets have size $h-1$, lie in $F$, and have equal sums, so
$F$ being $B_{h-1}$ forces $T_k=T_k'$, a contradiction. Hence $n_k=h$, so all
$h$ elements lie in block $k$ and $n_y=0$. The converse is clear. For the last
sentence, a multiset with all $h$ elements in block $k$ and parameter multiset
$T$ has weight $\nu_T r^h$ by \eqref{eq:wdef}.

(2) Immediate from (1): a set is $B_h[g]$ exactly when each block contributes a
$B_h[g]$ subset of $F$, and $y$ is unconstrained. Every point of $A_{F,N}$ has
probability $r$, so the largest $B_h[g]$ mass is $(1-q)+Nb_{h,g}(F)r$, and
$b_{h,g}(F)<\lvert F\rvert$ because $F$ is not $B_h[g]$.

(3) The law is $1-q$ on one atom and $r=q/(\lvert F\rvert N)$ on $\lvert F\rvert N$ atoms.

(4) By (1) the fibre structure of $s$ on $\Uc_h$ is $N$ disjoint copies of that
of $s$ on $\Uc_h(F)$, with all weights scaled by $r^h$, together with singleton
fibres, which contribute nothing to either side. Since the defect in
\eqref{eq:Lg} is a sum over fibres and $\chi$ may be chosen on each fibre
independently, the infimum decomposes into a sum over the $N$ blocks; by
homogeneity of degree $\alpha$ in the weights, each block contributes
$r^{h\alpha}\Rc^{[g]}_{\alpha,h}(F)$. The displayed formula for
$\Delta^{[g]}_{\alpha,h}$ is then \eqref{eq:LDelta}, solved for
$\Delta^{[g]}$; the sign $\varsigma$ records that
$M_\alpha((s,\chi)(U_h))$ lies below $M_\alpha(U_h)$ for $\alpha<1$ and above it
for $\alpha>1$.

(5) Identical, using \eqref{eq:shannonid} in place of \eqref{eq:Lg}: the
integrand vanishes off the fibres of (1), all atoms in a block have probability
$r$, and $\Lambda^{[g]}_h(F)$ is by definition the corresponding per-block
infimum for the weights $\nu_u$, which is homogeneous of degree $1$.
\end{proof}

\section{Sharpness: two templates across the diagram}\label{sec:lower}

Throughout this section the map of Sections~\ref{sec:coarse}
and~\ref{sec:upper} is specialized to the sum map $s$. The ambient group is $G=\Z$ except in Section~\ref{sec:dilution}, which also uses
general abelian groups, since $h$-torsion is what yields the exact dilution
constant. Parts 1--3 of Theorem~\ref{thm:phase} are all realized by the
separated-block family $Y_{F,N}^{(q)}$ of Example~\ref{ex:blocks}, with a
different parameter sent to its limit in each; part 3 does not require
$\alpha\ge\beta$, so that family reaches beyond the range $0<\beta\le\alpha$. The
dilution regime needs a construction of the opposite kind. The five limiting
procedures are summarized below:
\IfTwoCol{%
\begin{center}
\footnotesize
\setlength{\tabcolsep}{2pt}
\begin{tabular}{@{}lll@{}}
\hline
regime & tuning & scales\\
\hline
$h\alpha\le1$ & $N=1$, $q\downarrow0$
  & $C\asymp q^{h\alpha}$, $\delta\asymp q$\\
$h\alpha>1$, $\beta<1$ & $N\to\infty$
  & $q\asymp N^{-\frac{1-\beta}{\beta}}$,
    $C\asymp N^{-\frac{h\alpha-\beta}{\beta}}$\\
$\beta=1$ & $N\to\infty$, then $F$
  & $q\sim D/\log N$, $\log\tfrac1C\sim(h\alpha{-}1)\log N$\\
$\beta>1$ & $N\to\infty$
  & $q\to q_*>0$, $\delta\to q_*$\\
$\alpha<1$, $\beta>\alpha$ & dust, $N\to\infty$
  & $q\to0$, $\delta\to1-m_\beta(D)$\\
\hline
\end{tabular}
\end{center}}{%
\begin{center}
\setlength{\tabcolsep}{4pt}
\begin{tabular}{@{}lllll@{}}
\hline
regime & tuning & budget forces & deficit & deletion\\
\hline
$h\alpha\le1$ & $N=1$, $q\downarrow0$ & none
  & $C\asymp q^{h\alpha}$ & $\delta\asymp q$\\
$h\alpha>1$, $\beta<1$ & $N\to\infty$ & $q\asymp N^{-(1-\beta)/\beta}$
  & $C\asymp N^{-(h\alpha-\beta)/\beta}$ & $\delta\asymp q$\\
$\beta=1$ & $N\to\infty$, then $F$ & $q\sim D/\log N$
  & $\log\tfrac1C\sim(h\alpha-1)\log N$ & $\delta\sim q$\\
$\beta>1$ & $N\to\infty$ & $q\to q_*>0$
  & $C\to0$ & $\delta\to q_*$\\
$\alpha<1$, $\beta>\alpha$ & dust, $N\to\infty$ & $q\to0$
  & $C\to0$ & $\delta\to1-m_\beta(D)$\\
\hline
\end{tabular}
\end{center}}
In the critical row the two limits are iterated, not simultaneous: first
$N\to\infty$ with $F$ fixed, and only then $\lvert F\rvert\to\infty$. Whenever the
separated-block family is used, $F$ is a fixed finite set that is $B_{h-1}$ and
not $B_h[g]$, as supplied by Lemma~\ref{lem:BC}, and we put
\begin{equation}\label{eq:gamma}
  \gamma=\gamma_{h,g}(F):=1-\frac{b_{h,g}(F)}{\lvert F\rvert}\in(0,1] ,
\end{equation}
so that $\delta_{h,g}(Y_{F,N}^{(q)})=\gamma q$ by Lemma~\ref{lem:blocks}(2).

\subsection{The branch \texorpdfstring{$h\alpha\le1$}{h alpha <= 1}}\label{sec:subcrit}

\begin{proposition}\label{prop:lower1}
Let $h\ge2$, $g\ge1$, $h\alpha\le1$ and $D>0$. Take $N=1$ in
Example~\ref{ex:blocks} and write $Y^{(\varepsilon)}=Y_{F,1}^{(\varepsilon)}$.
Then $H_\beta(Y^{(\varepsilon)})\to0$ as $\varepsilon\downarrow0$ for every
$\beta>0$, $\delta_{h,g}(Y^{(\varepsilon)})=\gamma\varepsilon$, and
\[
  \Delta^{[g]}_{\alpha,h}\bigl(Y^{(\varepsilon)}\bigr)
  \ \sim\ \frac{\Rc^{[g]}_{\alpha,h}(F)}{(1-\alpha)\lvert F\rvert^{h\alpha}}\,
   \varepsilon^{h\alpha}\qquad(\varepsilon\downarrow0).
\]
Consequently $\Phi^{\Z,[g]}_{\alpha,\beta,D,h}(C)\gtrsim C^{1/(h\alpha)}$ for
all small $C>0$.
\end{proposition}

\begin{proof}
Since $h\ge2$ we have $\alpha\le\frac12<1$. By Lemma~\ref{lem:blocks}(3),
$P_\beta=(1-\varepsilon)^\beta+\lvert F\rvert^{1-\beta}\varepsilon^\beta\to1$ for
$\beta\ne1$, so $H_\beta=\log P_\beta/(1-\beta)\to0$, while
$H=\eta(\varepsilon)+\varepsilon\log \lvert F\rvert\to0$; in particular
$H_\beta(Y^{(\varepsilon)})\le D$ for all small $\varepsilon$, and
$\delta_{h,g}=\gamma\varepsilon$ by Lemma~\ref{lem:blocks}(2). The support is
finite, and $w_{\{y^h\}}=(1-\varepsilon)^h\to1$ while every other multiset has
weight $O(\varepsilon)$, so $M_\alpha(U_h)\to1$. Now
Lemma~\ref{lem:blocks}(4) with $N=1$, $r=\varepsilon/\lvert F\rvert$ gives
\IfTwoCol{%
\[
  \begin{aligned}
   &\Delta^{[g]}_{\alpha,h}\bigl(Y^{(\varepsilon)}\bigr)\\
   &\quad=\frac{1}{1-\alpha}\log\frac{M_\alpha(U_h)}
     {M_\alpha(U_h)-(\varepsilon/\lvert F\rvert)^{h\alpha}
       \Rc^{[g]}_{\alpha,h}(F)}\\
   &\quad\sim\frac{\Rc^{[g]}_{\alpha,h}(F)}
     {(1-\alpha)\lvert F\rvert^{h\alpha}}\,\varepsilon^{h\alpha},
  \end{aligned}
\]}{%
\[
  \Delta^{[g]}_{\alpha,h}\bigl(Y^{(\varepsilon)}\bigr)
  =\frac{1}{1-\alpha}\log\frac{M_\alpha(U_h)}
   {M_\alpha(U_h)-(\varepsilon/\lvert F\rvert)^{h\alpha}\Rc^{[g]}_{\alpha,h}(F)}
  \sim\frac{\Rc^{[g]}_{\alpha,h}(F)}{(1-\alpha)\lvert F\rvert^{h\alpha}}
     \,\varepsilon^{h\alpha},
\]}
using $\log(1+x)\sim x$ and $\Rc^{[g]}_{\alpha,h}(F)>0$ from
Lemma~\ref{lem:blockconst}. Eliminating $\varepsilon$ gives
$\delta_{h,g}\asymp C^{1/(h\alpha)}$ along this family, and the map
$\varepsilon\mapsto\Delta^{[g]}_{\alpha,h}(Y^{(\varepsilon)})$ is continuous on
$(0,1)$, being a minimum of finitely many continuous functions, and tends
to $0$, so every small value of $C$ is realized.
\end{proof}

\subsection{The branch \texorpdfstring{$h\alpha\ge1$}{h alpha >= 1} with a
subcritical budget}\label{sec:supcrit}

Here the light mass is spread over $N$ blocks and tuned so that the budget is
saturated. Put, for $0<\beta<1$,
\IfTwoCol{%
\begin{equation}\label{eq:frakM}
\begin{split}
  L&=L_{\beta,D}:=e^{(1-\beta)D}-1,\\
  \mathfrak M_{\alpha,\beta,D,h}&:=
  \begin{cases}
    \displaystyle\sum_{\ell=0}^{h}\frac{(h)_\ell^{\,\alpha}}{\ell!}L^{\ell},
      & \alpha=\beta,\\[4mm]
    1, & \beta<\alpha ,
  \end{cases}
\end{split}
\end{equation}}{%
\begin{equation}\label{eq:frakM}
  L=L_{\beta,D}:=e^{(1-\beta)D}-1,\qquad
  \mathfrak M_{\alpha,\beta,D,h}:=
  \begin{cases}
    \displaystyle\sum_{\ell=0}^{h}\frac{(h)_\ell^{\,\alpha}}{\ell!}L^{\ell},
      & \alpha=\beta,\\[4mm]
    1, & \beta<\alpha ,
  \end{cases}
\end{equation}}
where $(h)_\ell=h!/(h-\ell)!$; note $\mathfrak M\in[1,\infty)$ always. This is
the limiting value of $M_\alpha(U_h)$ along the family below, and it is nontrivial
only on the diagonal $\alpha=\beta$.

\begin{proposition}\label{prop:lower2}
Fix $h\ge2$, $g\ge1$, $D>0$ and $0<\beta<1$ with $\beta\le\alpha$ and
$h\alpha>1$. For all large $N$ there is $q_N$ with
$H_\beta(Y_{F,N}^{(q_N)})=D$, and then
\IfTwoCol{%
\[
  q_N\sim L^{1/\beta}(\lvert F\rvert N)^{-\frac{1-\beta}{\beta}},\qquad
  \begin{aligned}
    C_N&:=\Delta^{[g]}_{\alpha,h}\bigl(Y_{F,N}^{(q_N)}\bigr)\\
       &\sim c'\,N^{-\frac{h\alpha-\beta}{\beta}} ,
  \end{aligned}
\]}{%
\[
  q_N\sim L^{1/\beta}(\lvert F\rvert N)^{-\frac{1-\beta}{\beta}},
  \qquad
  C_N:=\Delta^{[g]}_{\alpha,h}\bigl(Y_{F,N}^{(q_N)}\bigr)
  \sim c'\,N^{-\frac{h\alpha-\beta}{\beta}} ,
\]}
where, for $\alpha\ne1$,
\begin{equation}\label{eq:cprime}
  c'=\frac{\Rc^{[g]}_{\alpha,h}(F)}
     {\lvert1-\alpha\rvert\,\mathfrak M_{\alpha,\beta,D,h}}
   \Bigl(\frac{L}{\lvert F\rvert}\Bigr)^{h\alpha/\beta}>0 ,
\end{equation}
and for $\alpha=1$ the same formula holds with
$\Rc^{[g]}_{\alpha,h}(F)/\lvert1-\alpha\rvert$ replaced by
$\Lambda^{[g]}_h(F)$. Consequently
$\delta_{h,g}(Y_{F,N}^{(q_N)})=\gamma q_N
\asymp C_N^{(1-\beta)/(h\alpha-\beta)}$, and
$\Phi^{\Z,[g]}_{\alpha,\beta,D,h}(C)\gtrsim C^{(1-\beta)/(h\alpha-\beta)}$ for
all small $C>0$.
\end{proposition}

\begin{proof}
\emph{Step 1: saturating the entropy budget.}
By Lemma~\ref{lem:blocks}(3) the map $q\mapsto P_\beta$ is continuous on
$[0,1]$, equals $1$ at $q=0$ and $(\lvert F\rvert N)^{1-\beta}$ at $q=1$, so for
$(\lvert F\rvert N)^{1-\beta}>e^{(1-\beta)D}>1$ there is $q_N$ with
$P_\beta=e^{(1-\beta)D}$. Then $q_N\to0$, $(1-q_N)^\beta\to1$ and
$(\lvert F\rvert N)^{1-\beta}q_N^\beta\to L$, giving the first asymptotic. With $r=q_N/(\lvert F\rvert N)$,
\begin{equation}\label{eq:diffasym}
  N r^{h\alpha}=N^{1-h\alpha}\lvert F\rvert^{-h\alpha}q_N^{h\alpha}
  \sim\Bigl(\frac{L}{\lvert F\rvert}\Bigr)^{h\alpha/\beta}
   N^{-\frac{h\alpha-\beta}{\beta}} ,
\end{equation}
because $q_N^{h\alpha}\sim L^{h\alpha/\beta}(\lvert F\rvert N)^{-h\alpha(1-\beta)/\beta}$ and
$h\alpha+h\alpha\frac{1-\beta}{\beta}=\frac{h\alpha}{\beta}$.

\emph{Step 2: the limiting power sum $M_\alpha(U_h)$.} If $\alpha>1$ then, since
$\sum_uw_u=1$ and $\alpha-1>0$,
\[
  (1-q_N)^{h\alpha}=w_{\{y^h\}}^{\,\alpha}\ \le\ M_\alpha(U_h)\ \le\
  \Bigl(\max_uw_u\Bigr)^{\alpha-1}\ \le\ 1 ,
\]
so $M_\alpha(U_h)\to1$. If $\alpha<1$ we split $\Uc_h$ according to the total
light multiplicity $\ell$ and the number $j\le\ell$ of \emph{distinct} light
atoms used. A multiset with $\ell$ light atoms all distinct has weight
$\frac{h!}{(h-\ell)!}(1-q_N)^{h-\ell}r^{\ell}=(h)_\ell(1+o(1))r^\ell$, and there
are $\binom{\lvert F\rvert N}{\ell}=(1+o(1))(\lvert F\rvert N)^\ell/\ell!$ of
them, so together they contribute
\IfTwoCol{%
\[
  \begin{aligned}
   &\frac{(h)_\ell^{\,\alpha}}{\ell!}
     (\lvert F\rvert N)^{\ell}r^{\ell\alpha}(1+o(1))\\
   &\quad=\frac{(h)_\ell^{\,\alpha}}{\ell!}
     \bigl[(\lvert F\rvert N)^{1-\alpha}q_N^{\alpha}\bigr]^{\ell}(1+o(1)),
  \end{aligned}
\]}{%
\[
  \frac{(h)_\ell^{\,\alpha}}{\ell!}(\lvert F\rvert N)^{\ell}r^{\ell\alpha}(1+o(1))
  =\frac{(h)_\ell^{\,\alpha}}{\ell!}
   \bigl[(\lvert F\rvert N)^{1-\alpha}q_N^{\alpha}\bigr]^{\ell}(1+o(1)),
\]}
using $r=q_N/(\lvert F\rvert N)$; and $(\lvert F\rvert N)^{1-\alpha}q_N^{\alpha}
\sim L^{\alpha/\beta}(\lvert F\rvert N)^{1-\alpha/\beta}$, which tends to $L$ if $\alpha=\beta$ and to $0$ if $\alpha>\beta$. A multiset with
$j<\ell$ distinct light atoms has weight of the same order $r^\ell$ but only
$O(N^{j})$ choices, so its class contributes $O(N^{j-\ell})$ relative to the
$j=\ell$ class and is negligible. Summing over $\ell=0,\dots,h$ gives
$M_\alpha(U_h)\to\mathfrak M_{\alpha,\beta,D,h}$ in both cases. The case
$\alpha=1$ needs no such computation.

\emph{Step 3: the deficit and the deletion cost.}
Lemma~\ref{lem:blocks}(4) for $\alpha\ne1$, respectively
Lemma~\ref{lem:blocks}(5) for $\alpha=1$, together with \eqref{eq:diffasym} and
$\log(1+x)\sim x$, yields $C_N\sim c'N^{-(h\alpha-\beta)/\beta}$ with $c'$ as
stated. Eliminating $N$ between
$\gamma q_N\asymp N^{-(1-\beta)/\beta}$ and
$C_N\asymp N^{-(h\alpha-\beta)/\beta}$ gives
$\gamma q_N\asymp C_N^{(1-\beta)/(h\alpha-\beta)}$.

\emph{Step 4: from $C_N$ to arbitrary $C$.} No monotonicity of the sequence
$(C_N)$ is needed. Since $C_N\sim c'N^{-(h\alpha-\beta)/\beta}$ with
$h\alpha>\beta$ we have $C_N\to0$ and $C_{N+1}/C_N\to1$. Fix $N_0$ large and,
given a small $C>0$, let $N$ be the least index $\ge N_0$ with $C_N\le C$; this
exists because $C_N\to0$, and for $C<C_{N_0}$ we have $N>N_0$, so minimality
gives $C<C_{N-1}$. Hence
\[
  1\ \ge\ \frac{C_N}{C}\ \ge\ \frac{C_N}{C_{N-1}}\ \longrightarrow\ 1 ,
\]
so $C_N=(1+o(1))C$; since $\Phi^{\Z,[g]}_{\alpha,\beta,D,h}$ is nondecreasing
and $Y_{F,N}^{(q_N)}$ is admissible for $C_N$, the claim follows.
\end{proof}

\begin{proof}[Proof of Theorem~\ref{thm:phase}(1)]
The upper bounds are Theorem~\ref{thm:general}(1),(2) with $f=s$, since
$1/(h\alpha)$ and $\frac{1-\beta}{h\alpha-\beta}$ are $\Theta_h(\alpha,\beta)$
in the respective ranges. Indeed, both denominators being positive,
\IfTwoCol{%
\begin{equation}\label{eq:thetacases}
\begin{split}
  \frac{1-\beta}{h\alpha-\beta}\ \gtrless\ \frac1{h\alpha}
  &\iff h\alpha(1-\beta)\gtrless h\alpha-\beta\\
  &\iff \beta\gtrless h\alpha\beta
   \iff 1\gtrless h\alpha ,
\end{split}
\end{equation}}{%
\begin{equation}\label{eq:thetacases}
  \frac{1-\beta}{h\alpha-\beta}\ \gtrless\ \frac1{h\alpha}
  \iff h\alpha(1-\beta)\gtrless h\alpha-\beta
  \iff \beta\gtrless h\alpha\beta
  \iff 1\gtrless h\alpha ,
\end{equation}}
which is the case distinction in \eqref{eq:theta}. The lower bounds are
Propositions~\ref{prop:lower1} and~\ref{prop:lower2}, whose exponents
$1/(h\alpha)$ and $(1-\beta)/(h\alpha-\beta)$ therefore match the upper bounds in
the respective branches. All constructions are integer-valued, so the conclusion
holds for $\Phi^{\Z,[g]}$ as well.
\end{proof}

\subsection{The Shannon boundary: exact logarithmic constant}\label{sec:crit}

At $\beta=1$ the budget no longer controls a power of the tail mass, only its
logarithm, and the polynomial rate degenerates. The construction is the same, but
now the block $F$ must also be sent to infinity: the fraction $\gamma$ of each
block that has to be deleted must tend to $1$, and this is exactly what
\eqref{eq:hierarchy} provides.

\begin{proof}[Proof of Theorem~\ref{thm:phase}(2), and of Theorem~\ref{thm:main}]
Theorem~\ref{thm:main} is the case $\alpha=\beta=1$, $h=2$, $g=1$. Let $\alpha\ge\beta=1$. The upper bound $\limsup\le(h\alpha-1)D$ is
\eqref{eq:critrate}.

\emph{Step 1: fix $F$ and choose $q_N$.} For the lower bound take $F=F^{(h)}_i$ from
Lemma~\ref{lem:BC}, with $\lvert F\rvert$ large enough that $b_{h,g}(F)<\lvert F\rvert$, so
that $F$ is not $B_h[g]$ and hence $\Rc^{[g]}_{\alpha,h}(F)>0$ and
$\Lambda^{[g]}_h(F)>0$ by Lemma~\ref{lem:blockconst}.

By Lemma~\ref{lem:blocks}(3) the map $q\mapsto\eta(q)+q\log(\lvert F\rvert N)$ is continuous,
vanishes at $q=0$ and equals $\log2+\frac12\log(\lvert F\rvert N)$ at $q=\frac12$, so for all
large $N$ there is $q_N\in(0,\frac12)$ with $H(Y_{F,N}^{(q_N)})=D$; from
$q_N\le D/\log(\lvert F\rvert N)$ we get $q_N\to0$, $\eta(q_N)=o(1)$ and
\begin{equation}\label{eq:qNbig}
  q_N=(1+o(1))\frac{D}{\log(\lvert F\rvert N)} .
\end{equation}
\emph{Step 2: compute $C_N$.}
Write $C_N=\Delta^{[g]}_{\alpha,h}(Y_{F,N}^{(q_N)})$ and $r=q_N/(\lvert F\rvert N)$. If
$\alpha=1$ then $C_N=\Lambda^{[g]}_h(F)q_N^{h}/(\lvert F\rvert^{h}N^{h-1})$ exactly, by
Lemma~\ref{lem:blocks}(5). If $\alpha>1$ then $M_\alpha(U_h)\to1$ by the
sandwich in the proof of Proposition~\ref{prop:lower2}, so
Lemma~\ref{lem:blocks}(4) gives
$C_N\sim\Rc^{[g]}_{\alpha,h}(F)Nr^{h\alpha}/(\alpha-1)$. In either case
$Nr^{h\alpha}=N^{1-h\alpha}(q_N/\lvert F\rvert)^{h\alpha}$ with $F$ fixed, so
\eqref{eq:qNbig} gives
\begin{equation}\label{eq:CNasym}
  C_N\ \sim\ K_F\,N^{-(h\alpha-1)}
   \bigl(\log(\lvert F\rvert N)\bigr)^{-h\alpha} ,
\end{equation}
with $K_F=\Lambda^{[g]}_h(F)(D/\lvert F\rvert)^{h}>0$ if $\alpha=1$ and
$K_F=\Rc^{[g]}_{\alpha,h}(F)(D/\lvert F\rvert)^{h\alpha}/(\alpha-1)>0$ if
$\alpha>1$. Taking logarithms, and using $\log(1/q_N)=O(\log\log N)$,
\IfTwoCol{%
\begin{equation}\label{eq:CNbig}
  \begin{aligned}
    \log\frac{1}{C_N}&=(h\alpha-1)\log N+O(\log\log N)\\
     &=(1+o(1))(h\alpha-1)\log(\lvert F\rvert N) .
  \end{aligned}
\end{equation}}{%
\begin{equation}\label{eq:CNbig}
  \log\frac{1}{C_N}=(h\alpha-1)\log N+O(\log\log N)
  =(1+o(1))(h\alpha-1)\log(\lvert F\rvert N) .
\end{equation}}
By Lemma~\ref{lem:blocks}(2), $\delta_{h,g}(Y_{F,N}^{(q_N)})=\gamma q_N$, so
\eqref{eq:qNbig} and \eqref{eq:CNbig} give
\begin{equation}\label{eq:alongCN}
  \delta_{h,g}\bigl(Y_{F,N}^{(q_N)}\bigr)\log\frac{1}{C_N}
  \longrightarrow\Bigl(1-\frac{b_{h,g}(F)}{\lvert F\rvert}\Bigr)(h\alpha-1)D .
\end{equation}

\emph{Step 3: from $C_N$ to all $C$.} Again no monotonicity of $(C_N)$ is needed.
By \eqref{eq:CNasym}, $C_N\to0$ and $C_{N+1}/C_N\to1$; the ratio
needs the asymptotic \eqref{eq:CNasym} rather than its logarithmic consequence
\eqref{eq:CNbig}, which leaves the ratio free to oscillate. Fix $N_0$
large; given small $C>0$, let $N$ be the least index $\ge N_0$ with $C_N\le C$,
which exists since $C_N\to0$, and for $C<C_{N_0}$ satisfies $C<C_{N-1}$ by
minimality. Then $1\ge C_N/C\ge C_N/C_{N-1}\to1$, so
$\log(1/C)=(1+o(1))\log(1/C_N)$, and since
$\Phi^{\Z,[g]}_{\alpha,1,D,h}$ is nondecreasing,
\[
  \liminf_{C\downarrow0}\Phi^{\Z,[g]}_{\alpha,1,D,h}(C)\log\frac1C
  \;\ge\;\Bigl(1-\frac{b_{h,g}(F)}{\lvert F\rvert}\Bigr)(h\alpha-1)D
\]
by \eqref{eq:alongCN}.

\emph{Step 4: the block grows.} Along the sequence of Lemma~\ref{lem:BC} the
ratio $b_{h,g}(F)/\lvert F\rvert$ tends to $0$ for each fixed $g$, so letting
$F$ run through it gives $\liminf\ge(h\alpha-1)D$. Finally $\Phi^{\Z,[g]}\le\Phi^{[g]}$ and the upper
bound \eqref{eq:critrate} applies to $\Phi^{[g]}$, so the same limit holds for
$\Phi^{[g]}_{\alpha,1,D,h}$ and for $\Phi^{\Z,[g]}_{\alpha,1,D,h}$.
\end{proof}

\subsection{The crossover at the Shannon budget}\label{sec:crossover}

\begin{proof}[Proof of Theorem~\ref{thm:crossover}]
Write $L=\log(1/C)$ and $\varepsilon=1-\beta$, so that $\varepsilon=(1+o(1))kx/L$;
since $\alpha\ge1>\beta$, the order condition $\beta\le\alpha$ holds throughout.

\emph{Upper bound.} Let $H_\beta(X)\le D$, that is $P_\beta\le e^{\varepsilon D}$.
Since $\sum_ap_a=1$,
\[
  P_\beta-1=\sum_ap_a\bigl(p_a^{-\varepsilon}-1\bigr)\ \le\ e^{\varepsilon D}-1 ,
\]
every term is nonnegative, and the terms with $p_a<\tau$ are at least
$p_a(\tau^{-\varepsilon}-1)$; hence
\begin{equation}\label{eq:crosstail}
  \PR(X\notin A_\tau)\ \le\ \frac{e^{\varepsilon D}-1}{\tau^{-\varepsilon}-1} .
\end{equation}
This replaces the bound $e^{\varepsilon D}\tau^{\varepsilon}$ of
Lemma~\ref{lem:tail}(1), which loses the cancellation in $P_\beta-1$. Now argue as
in the proof of Theorem~\ref{thm:general}(2): Lemma~\ref{lem:deletion} with
$\mathcal A=A_\tau$, \eqref{eq:mubound} and \eqref{eq:K} give
\[
  \delta_{h,g}(X)\ \le\ \frac{e^{\varepsilon D}-1}{\tau^{-\varepsilon}-1}
   +K_{\alpha,h}(C)\,\tau^{-k},
\]
where $K_{\alpha,h}(C)=O(C)$ does not depend on $\beta$ or $D$, by the last
sentence of Lemma~\ref{lem:mass}, because $\alpha\ge1$. Take
$\tau=(CL^2)^{1/k}$. Then $K_{\alpha,h}(C)\tau^{-k}=O(L^{-2})$,
$\tau^{-\varepsilon}=\exp\bigl(\varepsilon(L-2\log L)/k\bigr)\to e^x$, and
$e^{\varepsilon D}-1=(1+o(1))\varepsilon D=(1+o(1))kxD/L$. Hence
$\limsup(L/kD)\,\Phi^{[g]}_{\alpha,\beta,D,h}(C)\le x/(e^x-1)$.

\emph{Lower bound.} Fix a finite $F\subset\Z_{>0}$ that is $B_{h-1}$ and not
$B_h[g]$, as in Lemma~\ref{lem:BC}, and $\eta\in(0,1)$. Given $C$, put
$N=\lceil C^{-1/k}\rceil$ and $q=(1-\eta)\varepsilon D/(e^x-1)$, and let
$Y=Y^{(q)}_{F,N}$ be as in Example~\ref{ex:blocks}. By Lemma~\ref{lem:blocks}(3),
$P_\beta(Y)=(1-q)^{1-\varepsilon}+(\lvert F\rvert N)^{\varepsilon}q^{1-\varepsilon}$.
Since $q\asymp\varepsilon\asymp1/L$ we have $\varepsilon\log(1/q)\to0$, and
$\varepsilon\log(\lvert F\rvert N)=\varepsilon(L/k+O(1))\to x$; so
$q^{1-\varepsilon}=(1+o(1))q$, $(\lvert F\rvert N)^{\varepsilon}=e^x+o(1)$ and
$(1-q)^{1-\varepsilon}=1-q+O(q^2+\varepsilon q)$. Therefore
\[
  P_\beta(Y)=1+q(e^x-1)+o(\varepsilon)=1+(1-\eta)\varepsilon D+o(\varepsilon)
  <e^{\varepsilon D}
\]
for all small $C$, that is $H_\beta(Y)<D$. The deficit is $o(C)$: if $\alpha=1$,
Lemma~\ref{lem:blocks}(5) gives
$\Delta^{[g]}_{1,h}(Y)=\Lambda^{[g]}_h(F)\,q^h/(\lvert F\rvert^hN^{h-1})\le
\Lambda^{[g]}_h(F)\lvert F\rvert^{-h}q^h\,C$, because $N^{-k}\le C$; if
$\alpha>1$, the estimate in the proof of Theorem~\ref{thm:phase}(3) gives
$\Delta^{[g]}_{\alpha,h}(Y)\le\Rc^{[g]}_{\alpha,h}(F)(\alpha-1)^{-1}
\lvert F\rvert^{-h\alpha}(q/(1-q))^{h\alpha}N^{-k}=O_F(q^{h\alpha})\,C$. So
$Y$ is admissible for $C$, and by Lemma~\ref{lem:blocks}(2)
\[
\IfTwoCol{%
  \begin{aligned}
  \frac{L}{kD}\,\Phi^{\Z,[g]}_{\alpha,\beta,D,h}(C)\ &\ge\
  \frac{L}{kD}\,\gamma_{h,g}(F)\,q\\
  &\longrightarrow\ \gamma_{h,g}(F)\,(1-\eta)\,\frac{x}{e^x-1} .
  \end{aligned}}{%
  \frac{L}{kD}\,\Phi^{\Z,[g]}_{\alpha,\beta,D,h}(C)\ \ge\
  \frac{L}{kD}\,\gamma_{h,g}(F)\,q
  \ \longrightarrow\ \gamma_{h,g}(F)\,(1-\eta)\,\frac{x}{e^x-1} .}
\]
Letting $F$ run through the sequence of Lemma~\ref{lem:BC}, so that
$\gamma_{h,g}(F)\to1$, and then $\eta\downarrow0$ gives the matching lower bound,
over $\Z$ and hence in general.
\end{proof}

\begin{remark}[uniformity]\label{rem:crossunif}
Write $L=\log(1/C)$, $f(x)=x/(e^x-1)$ with $f(0)=1$, and
$F_C(x)=(L/kD)\,\Phi^{[g]}_{\alpha,1-kx/L,D,h}(C)$ for $0\le x<L/k$. The
convergence in Theorem~\ref{thm:crossover} is uniform in absolute error:
$\sup_{0\le x<L/k}\lvert F_C(x)-f(x)\rvert\to0$ as $C\downarrow0$. Indeed,
$H_\beta$ is nonincreasing in $\beta$, so lowering $\beta$ shrinks the set of
admissible laws and $F_C$ is nonincreasing in $x$; $F_C(0)\to1$ by
Theorem~\ref{thm:phase}(2) and $F_C(x)\to f(x)$ for each $x>0$ by
Theorem~\ref{thm:crossover}; and pointwise convergence of nonincreasing functions
to a continuous limit is uniform on each $[0,M]$, while for $x\ge M$ both
$F_C(x)\le F_C(M)$ and $f(x)\le f(M)$ are small once $M$ is large. The same holds
over $\Z$. The relative error is not uniform: at a fixed $\beta<1$,
Theorem~\ref{thm:phase}(1) gives $\Phi\asymp C^{\varepsilon/(k+\varepsilon)}$ with
$\varepsilon=1-\beta$, whereas $(kD/L)f(x)$ is of order
$\varepsilon C^{\varepsilon/k}$, and the ratio grows like
$C^{-\varepsilon^2/(k(k+\varepsilon))}$.
\end{remark}

As $x\to\infty$ the crossover decays like $xe^{-x}$; substituted formally at a
fixed $\beta<1$ this suggests $\Phi\approx(1-\beta)D\,C^{(1-\beta)/k}$, whose
exponent agrees with $(1-\beta)/(h\alpha-\beta)$ only to first order in
$1-\beta$, which is why the relative error above is not uniform.

\subsection{The supercritical budget: an exact instability floor}\label{sec:super}

For $\beta>1$ one has $H_\beta\le H$, so a bound on $H_\beta$ does not bound the
Shannon entropy and cannot control the mass on light atoms. The same template
exhibits the resulting floor: it suffices to freeze $q$ and let $N\to\infty$.

\begin{proof}[Proof of Theorem~\ref{thm:phase}(3)]
Let $\alpha\ge1$ and $\beta>1$, put $q_*=1-e^{-(\beta-1)D/\beta}\in(0,1)$, and take
$Y_N=Y_{F,N}^{(q_*)}$ from Example~\ref{ex:blocks}. By
Lemma~\ref{lem:blocks}(3), $P_\beta(Y_N)>(1-q_*)^\beta=e^{-(\beta-1)D}$, and
since $1/(1-\beta)<0$,
\[
  H_\beta(Y_N)=\frac{\log P_\beta}{1-\beta}
  <\frac{\log e^{-(\beta-1)D}}{1-\beta}=D .
\]
By Lemma~\ref{lem:blocks}(2), $\delta_{h,g}(Y_N)=\gamma q_*$ for every $N$. If
$\alpha>1$ then, by Lemma~\ref{lem:blocks}(4), $\log(1+x)\le x$ and
$M_\alpha(U_h)\ge w_{\{y^h\}}^\alpha=(1-q_*)^{h\alpha}$,
\[
  \Delta^{[g]}_{\alpha,h}(Y_N)\le
  \frac{\Rc^{[g]}_{\alpha,h}(F)}{(\alpha-1)\lvert F\rvert^{h\alpha}}
  \Bigl(\frac{q_*}{1-q_*}\Bigr)^{h\alpha}N^{1-h\alpha}\longrightarrow0,
\]
since $h\alpha\ge h>1$. If instead $\alpha=1$, then Lemma~\ref{lem:blocks}(5)
gives the exact value
\[
  \Delta^{[g]}_{1,h}(Y_N)=\Lambda^{[g]}_h(F)\,\frac{q_*^{h}}{\lvert F\rvert^{h}N^{h-1}}
  \ \longrightarrow\ 0 ,
\]
because $h\ge2$. Thus
$\Phi^{[g]}_{\alpha,\beta,D,h}(C)\ge\gamma_{h,g}(F)q_*$ for every $C>0$; letting
$F$ run through the sequence of Lemma~\ref{lem:BC} and using
$b_{h,g}(F)/\lvert F\rvert\to0$ gives $\gamma_{h,g}(F)\to1$ and hence
$\Phi^{[g]}_{\alpha,\beta,D,h}(C)\ge q_*$. All the $Y_N$ are integer-valued, so
the same holds for $\Phi^{\Z,[g]}$. The matching upper bound is
Theorem~\ref{thm:general}(4), which gives
$\Phi^{[g]}_{\alpha,\beta,D,h}(C)\le q_*+O(C^{\min\{1,(\beta-1)/(h\alpha-1)\}})$;
together these prove part 3, over $\Z$ as well as in general.
\end{proof}

By contrast, part 2 of Theorem~\ref{thm:phase} shows that lowering the budget
order to $1$ restores stability at every R\'enyi order $\alpha\ge1$. So on the
half-plane $\alpha\ge1$ it is $\beta=1$, and not $\alpha=1$, that bounds
stability; below $\alpha=1$ the relevant boundary is the diagonal itself, by
part 4.

\subsection{Dilution: the deficit off the range}\label{sec:dilution}

The constructions so far spread the light mass over blocks that all collide. The
remaining low-order regime $0<\alpha<1$, $\beta>\alpha$ is settled by a
construction of the opposite kind: keep \emph{one} weighted core that is badly
non-$B_h$, and add a large number of atoms in general position, called
\emph{dust}, placed so that they separate sums on their own; every collision that
survives therefore originates in the core, and Lemma~\ref{lem:dust} makes this
precise by tracking the multisets that use at least two core atoms. The dust does
not change the weight to be deleted, but for $\alpha<1$ it inflates
$M_\alpha(U_h)$, and the order-$\alpha$ deficit compares $M_\alpha(U_h)$ with
$M_\alpha(S_h)$; so the inflation dilutes the deficit to $0$.

\begin{example}[a dusted core]\label{ex:dust}
Let $G_0$ be an abelian group, let $A_0\subset G_0$ be finite and let $\rho$ be a
probability vector on $A_0$. Let $t_1<\dots<t_N$ be a $B_h$ set of positive
integers and let $B>h$. In $G=G_0\times\Z$ put
\[
  \mathcal C=\bigl\{(a,1):a\in A_0\bigr\},\qquad
  \mathcal T=\bigl\{(0,Bt_k):1\le k\le N\bigr\} ,
\]
and for $q\in(0,1)$ let $X=X^{(q)}_{\rho,N}$ have
$\PR\bigl(X=(a,1)\bigr)=(1-q)\rho_a$ for $a\in A_0$ and
$\PR\bigl(X=(0,Bt_k)\bigr)=q/N$ for each $k$. We call $\mathcal C$ the
\emph{core} and $\mathcal T$ the \emph{dust\nocorr}; the letter $D$ is reserved for the
budget.
\end{example}

\begin{lemma}\label{lem:dust}
Let $X=X^{(q)}_{\rho,N}$ be as in Example~\ref{ex:dust} and write
$\Psi:=N^{1-\alpha}q^{\alpha}$.
\begin{enumerate}
\item Two distinct multisets of $\Uc_h$ have equal sums if and only if they
contain the same number $j$ of core atoms, the same dust atoms, and their core
parts are distinct $j$-multisets from $A_0$ with equal sums in $G_0$. In
particular every fibre meeting a multiset with at most one core atom is a
singleton.
\item A subset $S\subseteq\mathcal C\cup\mathcal T$ is $B_h[g]$ if and only if
its core part is; consequently
\IfTwoCol{%
\[
  \begin{aligned}
    \delta_{h,g}(X)&=(1-q)\bigl(1-\mu_g(\rho)\bigr),\\
    \mu_g(\rho)&:=\max\bigl\{\rho(S_0):S_0\subseteq A_0\text{ is }B_h[g]\bigr\} .
  \end{aligned}
\]}{%
\[
  \delta_{h,g}(X)=(1-q)\bigl(1-\mu_g(\rho)\bigr),\qquad
  \mu_g(\rho):=\max\bigl\{\rho(S_0):S_0\subseteq A_0\text{ is }B_h[g]\bigr\} .
\]}
\item $P_\beta(X)=(1-q)^{\beta}\sum_a\rho_a^{\beta}+N^{1-\beta}q^{\beta}$ for
$\beta\ne1$, and $H(X)=(1-q)H(\rho)+\eta(q)+q\log N$.
\item For $0<\alpha<1$ there are constants $c,c'>0$ depending only on
$\rho,\alpha$ and $h$ such that, for all large $N$,
\IfTwoCol{%
\[
  \begin{aligned}
  &M_\alpha(U_h)\ \ge\ c\,\Psi^{h} ,\\
  &0\le M_\alpha(U_h)-M_\alpha(S_h)\ \le\ c'\sum_{j=2}^{h}\Psi^{\,h-j} ,
  \end{aligned}
\]}{%
\[
  M_\alpha(U_h)\ \ge\ c\,\Psi^{h}
  \qquad\text{and}\qquad
  0\le M_\alpha(U_h)-M_\alpha(S_h)\ \le\ c'\sum_{j=2}^{h}\Psi^{\,h-j} ,
\]}
so that, along any family with $\Psi\to\infty$,
$\Delta^{[g]}_{\alpha,h}(X)=O_{\rho,\alpha,h}\bigl(\Psi^{-2}\bigr)$; the last
assertion uses $\Psi\to\infty$ only to ensure
$M_\alpha(S_h)\ge\frac12c\,\Psi^{h}$.
\end{enumerate}
\end{lemma}

\begin{proof}
(1) The $\Z$-coordinate of the sum of a multiset with $j$ core atoms and dust
part $E$ is $j+B\sum_{d\in E}t_d$, and $0\le j\le h<B$, so that coordinate
determines $j$ and $\sum_{d\in E}t_d$; the $G_0$-coordinate is the sum of the
core part. Two multisets in one fibre therefore agree in $j$ and in
$\sum_{d\in E}t_d$; as $\{t_k\}$ is a $B_h$ set it is also a $B_{h-j}$ set (pad
two $(h-j)$-multisets with $j$ copies of $t_1$), so their dust parts agree,
and the core parts are $j$-multisets from $A_0$ with equal $G_0$-sums. If $j\le1$
a $j$-multiset is determined by its sum, so the fibre is a singleton. The
converse is clear.

(2) By (1) the fibre of $s$ through a multiset with $j$ core atoms and dust part
$E$ consists of the multisets obtained by replacing its core part by another
$j$-multiset from $S\cap\mathcal C$ with the same sum. Hence $S$ is $B_h[g]$ if
and only if, for every $j\le h$, no $j$-fold sum from the core part of $S$ has
more than $g$ representations, that is, if and only if the core part of $S$ is
$B_j[g]$ for every $j\le h$. Padding with $h-j$ copies of a fixed element shows
that $B_h[g]$ implies $B_j[g]$, so this says exactly that the core part is
$B_h[g]$. The dust may therefore always be kept in full, and the largest
$B_h[g]$ weight is $q+(1-q)\mu_g(\rho)$.

(3) The law is $(1-q)\rho_a$ on $\lvert A_0\rvert$ atoms and $q/N$ on $N$ atoms.

(4) \emph{Lower bound for $M_\alpha(U_h)$.}
The $\binom{N}{h}$ multisets of $h$ distinct dust atoms have weight
$h!\,(q/N)^h$ each, so
$M_\alpha(U_h)\ge\binom{N}{h}\bigl(h!(q/N)^h\bigr)^{\alpha}
\sim\frac{(h!)^{\alpha}}{h!}\Psi^{h}$, which is the lower bound for
$M_\alpha(U_h)$.

\emph{Upper bound for the fibre defect.} The defect is a sum of nonnegative
fibrewise terms, so it is at most the
total $\alpha$-mass of the multisets lying in nonsingleton fibres, and by (1)
those have $j\ge2$ core atoms. A multiset with $j$ core atoms has weight at most
$h!\,(q/N)^{h-j}$ times a constant depending on $\rho$ and $h$, and there are at
most $\lvert A_0\rvert^{j}\binom{N+h-j-1}{h-j}=O_{h,\rho}(N^{h-j})$ of them;
raising to the power $\alpha$ and summing gives
$O\bigl(N^{(h-j)(1-\alpha)}q^{(h-j)\alpha}\bigr)=O(\Psi^{h-j})$ for each
$j\ge2$. Finally, by \eqref{eq:LDelta},
$\Delta^{[g]}_{\alpha,h}(X)\le\Delta_{\alpha,h}(X)
=\frac{1}{1-\alpha}\log\frac{M_\alpha(U_h)}{M_\alpha(S_h)}
\le\frac{1}{1-\alpha}\cdot\frac{M_\alpha(U_h)-M_\alpha(S_h)}{M_\alpha(S_h)}$,
and $M_\alpha(S_h)\ge\frac12 c\Psi^{h}$ for large $N$, so the ratio is
$O(\Psi^{-2})$.
\end{proof}

\begin{proposition}[dilution]\label{prop:dilute}
Let $h\ge2$, $g\ge1$, $D>0$, and let $0<\alpha<1$ and $\beta>\alpha$. Let $A_0$,
$\rho$ be as in Example~\ref{ex:dust} with $H_\beta(\rho)<D$. Then there is
$q=q(N)\to0$ such that $H_\beta(X^{(q)}_{\rho,N})\le D$ for all large $N$ and
$\Psi=N^{1-\alpha}q^{\alpha}\to\infty$; consequently
\[
  \Delta^{[g]}_{\alpha,h}\bigl(X^{(q)}_{\rho,N}\bigr)\longrightarrow0,
  \qquad
  \delta_{h,g}\bigl(X^{(q)}_{\rho,N}\bigr)\longrightarrow1-\mu_g(\rho) ,
\]
and therefore $\Phi^{[g]}_{\alpha,\beta,D,h}(C)\ge1-\mu_g(\rho)$ for every
$C>0$.
\end{proposition}

\begin{proof}
The two displayed limits follow from Lemma~\ref{lem:dust}(2),(4) once $q\to0$ and
$\Psi\to\infty$, and they give the last assertion because for each fixed $C>0$ the
variable $X^{(q(N))}_{\rho,N}$ is admissible for $C$ as soon as $N$ is large. It
remains to choose $q$.

\emph{Case $\beta<1$.} Put $\Sigma=\sum_a\rho_a^{\beta}=e^{(1-\beta)H_\beta(\rho)}$, so
$\eta:=e^{(1-\beta)D}-\Sigma>0$, and take
$q=\eta^{1/\beta}N^{-(1-\beta)/\beta}$. Then $q\to0$ and, by
Lemma~\ref{lem:dust}(3), $P_\beta(X)\le\Sigma+\eta=e^{(1-\beta)D}$, i.e.\
$H_\beta(X)\le D$; and
$\Psi=\eta^{\alpha/\beta}N^{1-\alpha/\beta}\to\infty$ because $\alpha<\beta$.

\emph{Case $\beta=1$.} Choose $c$ with $0<c<D-H(\rho)$ and put $q=c/\log N$. Then $q\to0$
and, by Lemma~\ref{lem:dust}(3), $H(X)\le H(\rho)+\eta(q)+c\le D$ for all large
$N$; and $\Psi=c^{\alpha}N^{1-\alpha}(\log N)^{-\alpha}\to\infty$ because
$\alpha<1$.

\emph{Case $\beta>1$.} Here $\Sigma=\sum_a\rho_a^\beta=e^{(1-\beta)H_\beta(\rho)}
>e^{-(\beta-1)D}$ since $1-\beta<0$, so there is $q_0>0$ with
$(1-q_0)^{\beta}\Sigma\ge e^{-(\beta-1)D}$. Fix
$\theta\in\bigl(0,\frac{1-\alpha}{\alpha}\bigr)$ and put
$q=\min\{q_0,N^{-\theta}\}$. Then $q\to0$ and
$P_\beta(X)\ge(1-q)^{\beta}\Sigma\ge e^{-(\beta-1)D}$, which for $\beta>1$ is
$H_\beta(X)\le D$. Moreover $N^{-\theta}<q_0$ for all large $N$, so that
$q=N^{-\theta}$ eventually and hence
$\Psi=N^{1-\alpha-\theta\alpha}\to\infty$, the exponent being positive because
$\theta<(1-\alpha)/\alpha$.
\end{proof}

Two choices of core now give the two statements we need. The first keeps
everything inside $\Z$.

\begin{corollary}[no stability over $\Z$ off the range]\label{cor:dilZ}
Let $h\ge2$, $g\ge1$, $D>0$, $0<\alpha<1$ and $\beta>\alpha$. Then there is
$c=c(\alpha,\beta,D,h,g)>0$ with
$\Phi^{\Z,[g]}_{\alpha,\beta,D,h}(C)\ge c$ for every $C>0$.
\end{corollary}

\begin{proof}
Take $G_0=\Z$ and $A_0=F\subset\{1,\dots,M\}$ finite, $B_{h-1}$ and not
$B_h[g]$, as in Lemma~\ref{lem:BC}, and let $\rho$ put weight $1-(\lvert F\rvert-1)\varepsilon$
on one point of $F$ and $\varepsilon$ on each of the other $\lvert F\rvert-1$, with
$\varepsilon>0$ so small that $H_\beta(\rho)<D$. Since $F$ is not $B_h[g]$, no
$B_h[g]$ subset of $F$ is all of $F$, so $\mu_g(\rho)\le1-\varepsilon$ and
Proposition~\ref{prop:dilute} gives the bound with $c=\varepsilon$.

For the construction to live in $\Z$, realize Example~\ref{ex:dust} in base $B$
instead of in $\Z\times\Z$: with $B>hM+h\max_kt_k$ replace $(a,1)$ by
$aB^{0}+B^{1}$ and $(0,Bt_k)$ by $t_kB^{2}+B^{3}$. No carry occurs, so the
base-$B$ digits in positions $0,1,2,3$ of a sum return the core parameter sum,
the number of core atoms, the dust parameter sum and the number of dust atoms;
this is exactly the information used in the proof of Lemma~\ref{lem:dust}(1), so
that lemma and everything after it hold verbatim.
\end{proof}

The second choice uses torsion, and gives the exact value.

\begin{corollary}[the exact floor for $\alpha<\beta$]\label{cor:dilexact}
Let $h\ge2$, $D>0$, $0<\alpha<1$ and $\beta>\alpha$. Then
\[
  \Phi_{\alpha,\beta,D,h}(C)\;=\;1-m_\beta(D)
  \qquad\text{for every }C>0 .
\]
\end{corollary}

\begin{proof}
The inequality $\le$ is Proposition~\ref{prop:ceiling}, for every $C$. For the
reverse, fix $m>m_\beta(D)$ with $1/m\notin\Z$, which is no restriction since we
shall let $m\downarrow m_\beta(D)$, so that $H_\beta(v_m)<D$ by
Lemma~\ref{lem:mbeta}, and let $n=\lfloor1/m\rfloor$. Then $v_m$ has exactly
$n+1$ positive entries. Take $G_0=(\Z/h\Z)^{r}$ with $h^{r}\ge n+1$, let $A_0$
consist of $n+1$ distinct points of $G_0$, and let $\rho=v_m$. Any two distinct $a,b\in G_0$ satisfy
$ha=hb=0$, so the multisets $\{a^{h}\}$ and $\{b^{h}\}$ have equal sums and no
$B_h$ subset of $A_0$ has two elements; hence
$\mu_1(\rho)=\lVert v_m\rVert_\infty=m$. Now
Proposition~\ref{prop:dilute} gives $\Phi_{\alpha,\beta,D,h}(C)\ge1-m$ for every
$C>0$, and letting $m\downarrow m_\beta(D)$ finishes the proof.
\end{proof}

\begin{remark}\label{rem:torsion}
Corollary~\ref{cor:dilexact} is the one place where the group matters: its lower
bound uses $h$-torsion, and over $\Z$ every two-element set is a $B_h$ set, since
for $a\ne b$ the sums $ia+(h-i)b$, $0\le i\le h$, are distinct. So the argument
cannot be run there, and over $\Z$ more mass can be kept: the next corollary
computes the integer value for $\beta=2$ and shows that it is strictly smaller.
\end{remark}

\begin{corollary}[the integer value at budget order $2$]\label{cor:Zbeta2}
Let $h\ge2$, $0<\alpha<1$ and $D>0$. Then for every $C>0$
\[
  1-e^{-D/2}\ \le\ \Phi^{\Z}_{\alpha,2,D,h}(C)\ \le\
  1-\min\bigl\{e^{-D/2},\,2e^{-D}\bigr\} ;
\]
in particular $\Phi^{\Z}_{\alpha,2,D,h}(C)=1-e^{-D/2}$ for $0<D\le\log4$. Since
$m_2(D)<e^{-D/2}$ for every $D>0$, on that range
$\Phi^{\Z}_{\alpha,2,D,h}(C)<\Phi_{\alpha,2,D,h}(C)=1-m_2(D)$; at $D=\log2$ the two
values are $1-1/\sqrt2$ and $\frac12$. More generally,
$\Phi^{\Z}_{\alpha,\beta,D,h}(C)\ge1-e^{-(\beta-1)D/\beta}$ for every $\beta>1$
and every $C>0$.
\end{corollary}

\begin{proof}
\emph{Upper bound.} Let $X$ be integer-valued with $H_2(X)\le D$, that is
$P_2=\sum_ap_a^2\ge e^{-D}$, let $p_1\ge p_2$ be its two largest atoms
($p_2=0$ if there is only one), and put $t=p_1+p_2$. A two-element subset of $\Z$
is a $B_h$ set by Remark~\ref{rem:torsion}, so $\delta_h(X)\le1-t$. Since every
atom other than the largest is at most $p_2$,
\[
  e^{-D}\le P_2\le p_1^2+p_2\sum_{i\ge2}p_i=p_1^2+(t-p_1)(1-p_1)=:\phi(p_1).
\]
Since $\phi$ is convex and $t/2\le p_1\le t$, the value $\phi(p_1)$ is at most
the larger of the endpoint values $\phi(t/2)=t/2$ and $\phi(t)=t^2$. Hence
$t\ge\min\{e^{-D/2},2e^{-D}\}$, which is the upper bound, uniformly in $C$. For
$D\le\log4$ one has $e^{-D/2}\ge\frac12$, that is $2e^{-D}\ge e^{-D/2}$, so the
minimum is $e^{-D/2}$.

\emph{Lower bound.} Let $\beta>1$ and put $q_*=1-e^{-(\beta-1)D/\beta}$. Take the
law $\rho=Y^{(q_*)}_{F,1}$ of Example~\ref{ex:blocks} with $N=1$, which lives on a
finite set $A_0\subset\Z_{>0}$ consisting of the heavy atom and one block. As in
the proof of Theorem~\ref{thm:phase}(3), $P_\beta(\rho)>(1-q_*)^\beta
=e^{-(\beta-1)D}$, that is $H_\beta(\rho)<D$, and by Lemma~\ref{lem:blocks}(2)
the largest $B_h$ mass under $\rho$ is
$\mu_1(\rho)=1-\gamma_{h,1}(F)\,q_*$. Now run Proposition~\ref{prop:dilute} with
this core, realized inside $\Z$ in base $B$ exactly as in the proof of
Corollary~\ref{cor:dilZ}; the case $\beta>1$ of that proposition gives
$\Phi^{\Z}_{\alpha,\beta,D,h}(C)\ge1-\mu_1(\rho)=\gamma_{h,1}(F)\,q_*$ for every
$C>0$. Letting $F$ run through the sequence of Lemma~\ref{lem:BC}, so that
$\gamma_{h,1}(F)\to1$, gives $\Phi^{\Z}_{\alpha,\beta,D,h}(C)\ge q_*$, and
$\beta=2$ gives $q_*=1-e^{-D/2}$.

\emph{Comparison.} Let $m=m_2(D)<1$. By Lemma~\ref{lem:mbeta} the vector $v_m$
has $H_2(v_m)=D$, so $\sum_iv_{m,i}^2=e^{-D}$, and it has at least two positive
entries, so $e^{-D}>m^2$. Hence $m_2(D)<e^{-D/2}$, and
Corollary~\ref{cor:dilexact} gives $\Phi_{\alpha,2,D,h}=1-m_2(D)>1-e^{-D/2}$. At
$D=\log2$, $m_2(\log2)=\frac12$ by Lemma~\ref{lem:mbeta} and $e^{-D/2}=1/\sqrt2$.
\end{proof}

For large budgets the heavy atom is no longer the worst core over $\Z$. A uniform
core on an interval forces more deletion, because large $B_h$ subsets of an
interval are rare.

\begin{proposition}[the floor is not attained for large budgets]\label{prop:Zlarge}
Let $h\ge2$, $0<\alpha<1$ and $1<\beta<h$, and put
\[
  D_0=\max\Bigl\{\log2,\
  \frac{\log\bigl((h\cdot h!)^{1/h}\,2^{(h-1)/h}\bigr)}
       {(h-1)/h-(\beta-1)/\beta}\Bigr\}.
\]
Then for every $D\ge D_0$ and every $C>0$, with $n=\lceil e^D\rceil-1$,
\[
  \Phi^{\Z}_{\alpha,\beta,D,h}(C)\ \ge\ 1-\frac{b_h(\{1,\dots,n\})}{n}
  \ >\ 1-e^{-(\beta-1)D/\beta} .
\]
For example, for $h=3$, $\beta=2$ and $D=\log401$ one has
$\Phi^{\Z}_{\alpha,2,\log401,3}(C)\ge1-\frac{18}{400}=0.955$, whereas
$1-e^{-D/2}=1-401^{-1/2}\approx0.9501$.
\end{proposition}

\begin{proof}
Let $\rho$ be uniform on $A_0=\{1,\dots,n\}$. Then $n<e^D$, so
$H_\beta(\rho)=\log n<D$, and the largest $B_h$ mass under $\rho$ is
$\mu_1(\rho)=b_h(A_0)/n$. Proposition~\ref{prop:dilute}, realized inside $\Z$ as in
the proof of Corollary~\ref{cor:dilZ}, gives the first inequality. For the second, let
$b=b_h(A_0)\ge1$. The $h$-fold sums of a $B_h$ subset of $A_0$ are distinct
elements of $\{h,\dots,hn\}$, so
$b^h/h!<\binom{b+h-1}{h}\le h(n-1)+1<hn$, that is $b<(h\cdot h!\,n)^{1/h}$
strictly. Since $n\ge e^D-1\ge e^D/2$ because $D\ge\log2$,
\[
\IfTwoCol{%
  \begin{aligned}
  \frac{b}{n}\ &<\ (h\cdot h!)^{1/h}\,n^{-(h-1)/h}\\
  &\le\ (h\cdot h!)^{1/h}\,2^{(h-1)/h}\,e^{-(h-1)D/h}
  \ \le\ e^{-(\beta-1)D/\beta},
  \end{aligned}}{%
  \frac{b}{n}\ <\ (h\cdot h!)^{1/h}\,n^{-(h-1)/h}\ \le\
  (h\cdot h!)^{1/h}\,2^{(h-1)/h}\,e^{-(h-1)D/h}
  \ \le\ e^{-(\beta-1)D/\beta},}
\]
the last step because $(h-1)/h>(\beta-1)/\beta$, which is $\beta<h$, and
$D\ge D_0$; the first inequality is strict, so the conclusion holds for
$D=D_0$ as well. In the example the threshold is not needed: the $3$-fold sums of a
$B_3$ subset $S\subseteq\{1,\dots,400\}$ are distinct elements of
$\{3,\dots,1200\}$, so $\binom{\lvert S\rvert+2}{3}\le1198<\binom{21}{3}$ and
$\lvert S\rvert\le18$.
\end{proof}

Corollaries~\ref{cor:dilZ} and~\ref{cor:dilexact} are Theorem~\ref{thm:phase}(4),
which completes the proof of Theorem~\ref{thm:phase}; its only $g$-dependent inputs
are $\Rc^{[g]}_{\alpha,h}(F)>0$ and $\Lambda^{[g]}_h(F)>0$, which hold as soon as
$F$ is not $B_h[g]$ (Lemma~\ref{lem:blockconst}), and \eqref{eq:hierarchy}.

\section{Finite sets and moments of the representation function}\label{sec:energy}

We now read the results off in purely finite terms. Throughout this section
$A\subset G$ is finite with $\lvert A\rvert=n$ and $X\sim\Unif(A)$, so that
$H(X)=\log n$, $H_\alpha(X)=\log n$ for every $\alpha$, and
\begin{equation}\label{eq:unifd}
  \delta_{h,g}(X)=1-\frac{b_{h,g}(A)}{n} .
\end{equation}
Write $r_h(z)=\#\{(a_1,\dots,a_h)\in A^h:\sum_ia_i=z\}$ for the ordered
representation function, so that $q_z=r_h(z)/n^h$ and $w_u=\nu_u/n^h$, and for
$\alpha>0$ put
\begin{equation}\label{eq:alphaenergy}
  E^{(\alpha)}_h(A):=\sum_z r_h(z)^{\alpha},
  \qquad
  \Ec^{(\alpha)}_h(n):=\sum_{u\in\Uc_h(A)}\nu_u^{\alpha},
\end{equation}
where $\Ec^{(\alpha)}_h(n)$ depends only on $n$ and $h$. For $\alpha=2$ these are the additive
energy $E_h(A)=E^{(2)}_h(A)$, the number of $2h$-tuples with
$a_1+\dots+a_h=b_1+\dots+b_h$, and $\Ec_h(n)=\Ec^{(2)}_h(n)$, the number of
pairs of ordered $h$-tuples that are rearrangements of one another; so
$\Ec_2(n)=2n^2-n$.

\begin{proposition}[Representation-moment identity]\label{prop:alphaenergy}
Let $\alpha>0$, $\alpha\ne1$. Then
\IfTwoCol{%
\[
  \begin{aligned}
  H_\alpha(S_h)&=\frac{h\alpha\log n-\log E^{(\alpha)}_h(A)}{\alpha-1},\\
  H_\alpha(U_h)&=\frac{h\alpha\log n-\log \Ec^{(\alpha)}_h(n)}{\alpha-1},
  \end{aligned}
\]}{%
\[
  H_\alpha(S_h)=\frac{h\alpha\log n-\log E^{(\alpha)}_h(A)}{\alpha-1},\qquad
  H_\alpha(U_h)=\frac{h\alpha\log n-\log \Ec^{(\alpha)}_h(n)}{\alpha-1},
\]}
and consequently
\begin{equation}\label{eq:Deltaalphaenergy}
  \Delta_{\alpha,h}(X)
  =\frac{1}{1-\alpha}\,\log\frac{\Ec^{(\alpha)}_h(n)}{E^{(\alpha)}_h(A)} .
\end{equation}
In particular $E^{(\alpha)}_h(A)\le\Ec^{(\alpha)}_h(n)$ for $\alpha<1$ and
$E^{(\alpha)}_h(A)\ge\Ec^{(\alpha)}_h(n)$ for $\alpha>1$, with equality in
either case if and only if $A$ is a $B_h$ set.
\end{proposition}

\begin{proof}
Since $q_z=r_h(z)/n^h$ and $w_u=\nu_u/n^h$ we have
$M_\alpha(S_h)=n^{-h\alpha}E^{(\alpha)}_h(A)$ and
$M_\alpha(U_h)=n^{-h\alpha}\Ec^{(\alpha)}_h(n)$. Now apply the identity
$H_\alpha=(1-\alpha)^{-1}\log M_\alpha$ and subtract; the sign statements and
the equality case are Proposition~\ref{prop:listzero} with $g=1$.
\end{proof}

\begin{proposition}[Additive-energy identity]\label{prop:energy}
If $X\sim\Unif(A)$ then $H_2(S_h)=2h\log n-\log E_h(A)$ and
$H_2(U_h)=2h\log n-\log\Ec_h(n)$, hence
$\Delta_{2,h}(X)=\log\bigl(E_h(A)/\Ec_h(n)\bigr)$. In particular
$E_h(A)\ge\Ec_h(n)$, with equality if and only if $A$ is $B_h$.
\end{proposition}

\begin{proof}
This is Proposition~\ref{prop:alphaenergy} at $\alpha=2$, where
$H_2=-\log M_2$.
\end{proof}

So the deficit at order $\alpha$ measures exactly how far the $\alpha$-th moment
of the representation function is from the value it takes on a $B_h$ set, and
letting $\alpha$ range over $(0,\infty)$ probes all nontrivial positive moments.
The classical additive energy is the single moment $\alpha=2$. Note also that
$\Rc^{[1]}_{\alpha,h}(A)=\bigl\lvert\Ec^{(\alpha)}_h(n)-E^{(\alpha)}_h(A)
\bigr\rvert$, so \eqref{eq:Rg} is the $g$-split refinement of the same moment
gap.

\subsection{Finite deletion bounds}

For uniform weights the deletion principle needs no moment at all. Let
$X\sim\Unif(A)$ and $f=s$, so that $w_u=\nu_u/n^h$, and write $I_g(A)$ for the set
\eqref{eq:Ig}: the multisets that are not among the $g$ multisets of largest
multiplicity $\nu_u$ in their fibre. With $\bar r_h(z)$ the number of multisets
$u\in\Uc_h(A)$ with $s(u)=z$, its size is $\lvert I_g(A)\rvert=\sum_z
(\bar r_h(z)-g)_+$.

\begin{proposition}[Finite deletion at every order]\label{prop:finiteall}
Let $A\subset G$ be finite with $\lvert A\rvert=n$, and let $h\ge2$, $g\ge1$.
Then for every $\alpha>0$ with $\alpha\ne1$,
\[
  n-b_{h,g}(A)\ \le\ \lvert I_g(A)\rvert\ \le\
  \frac{\Rc^{[g]}_{\alpha,h}(A)}{\lvert2-2^\alpha\rvert} ,
\]
and the right-hand side tends to $\lvert I_g(A)\rvert$ as $\alpha\downarrow0$.
\end{proposition}

\begin{proof}
Every atom of $X$ has probability $1/n$, so $m_u=1/n$ for every $u$, and
Lemma~\ref{lem:deletion} with $\mathcal A=A$ gives a $B_h[g]$ set $B$ with
$(n-\lvert B\rvert)/n\le\lvert I_g(A)\rvert/n$. Lemma~\ref{lem:blockconst},
applied to $A$ in place of $F$, gives
$\Rc^{[g]}_{\alpha,h}(A)\ge d_\alpha\sum_z\sum_{j>g}\nu_{z,j}^{\,\alpha}\ge
d_\alpha\lvert I_g(A)\rvert$, because $\nu_u\ge1$. For the limit, the infimum in
\eqref{eq:Rg} splits over the finitely many fibres, as in the proof of
Lemma~\ref{lem:blocks}(4), and on each fibre it is a minimum over finitely many
partitions into at most $g$ classes. As $\alpha\downarrow0$ every $\nu_u^\alpha$
and every class sum raised to the power $\alpha$ tends to $1$, so the term of a
partition tends to the number of multisets in the fibre minus the number of
classes, whose minimum is $(\bar r_h(z)-g)_+$; and $d_\alpha\to1$.
\end{proof}

\begin{remark}[where the order $1/h$ matters]\label{rem:regimes}
Proposition~\ref{prop:finiteall} is elementary: deleting one point from each
excess multiset is always enough, and for uniform weights every moment gap
dominates the number of excess multisets because all multiplicities are at least
$1$. Specializing Theorem~\ref{thm:rawmoment} to uniform weights instead gives
$n-b_{h,g}(A)\le(\Rc^{[g]}_{\alpha,h}(A)/d_\alpha)^{1/(h\alpha)}$ for
$0<\alpha\le1/h$, since $\Lc_{\alpha,g}(s)=n^{-h\alpha}\Rc^{[g]}_{\alpha,h}(A)$;
this agrees with Proposition~\ref{prop:finiteall} at $\alpha=1/h$ and is weaker
for $\alpha<1/h$, because $\Rc^{[g]}_{\alpha,h}(A)/d_\alpha\ge1$ as soon as $A$
is not $B_h[g]$. The condition $h\alpha\le1$ is therefore critical for the
weighted problem, where atoms of different sizes force the subadditivity step in
the proof of Theorem~\ref{thm:rawmoment}, and not for uniform sets. Nor are the
finite bounds strong for highly additive sets: for $A=\{1,\dots,n\}$ both
$\lvert I_g(A)\rvert$ and $\Rc^{[g]}_{\alpha,h}(A)/d_\alpha$ exceed the trivial
bound $n$. Their content is for sets with few collisions, where they are sharp in
order (Proposition~\ref{prop:momentsharp}).
\end{remark}

At $\alpha=2$ the moment gap is the excess additive energy.

\begin{corollary}[A second-moment deletion bound]\label{cor:combin}
Every finite $A\subset G$ with $\lvert A\rvert=n$ has a $B_h[g]$ subset $B$ with
$\lvert A\rvert-\lvert B\rvert\le\bigl(E_h(A)-\Ec_h(n)\bigr)/2$.
\end{corollary}

\begin{proof}
Here $d_2=2$, and for every labelling $\chi$ the classes of $(s,\chi)$ refine the
fibres of $s$, so by superadditivity of $t\mapsto t^2$ the gap in \eqref{eq:Rg}
for $\chi$ is at most the gap for the constant labelling; hence
$\Rc^{[g]}_{2,h}(A)\le\Rc^{[1]}_{2,h}(A)=E_h(A)-\Ec_h(n)$ by
Proposition~\ref{prop:energy}. Now apply Proposition~\ref{prop:finiteall} with
$\alpha=2$.
\end{proof}

\begin{proposition}[Sharpness in order]\label{prop:momentsharp}
Fix $h\ge2$ and $g\ge1$ and let $F$ be $B_{h-1}$ and not $B_h[g]$, with
$B>hM$ as in Example~\ref{ex:blocks}. For $N\ge1$ put $A_N=A_{F,N}$, so that
$\lvert A_N\rvert=\lvert F\rvert N$. Then for every $\alpha>0$, $\alpha\ne1$,
\[
  \begin{aligned}
    \lvert A_N\rvert-b_{h,g}(A_N)&=\gamma_{h,g}(F)\,\lvert A_N\rvert ,\\
    \lvert I_g(A_N)\rvert&=\frac{\lvert I_g(F)\rvert}{\lvert F\rvert}\,
      \lvert A_N\rvert ,\\
    \Rc^{[g]}_{\alpha,h}(A_N)
      &=\frac{\Rc^{[g]}_{\alpha,h}(F)}{\lvert F\rvert}\,\lvert A_N\rvert .
  \end{aligned}
\]
In particular all three quantities in Proposition~\ref{prop:finiteall} are
$\asymp_{h,g,F,\alpha}\lvert A_N\rvert$, so the linear dependence there cannot be
improved to any smaller power.
\end{proposition}

\begin{proof}
By Lemma~\ref{lem:blocks}(1) applied with the heavy atom removed, which changes
nothing since no collision involves $y$, the nonsingleton fibres of
$s$ on $\Uc_h(A_N)$ are $N$ disjoint copies of those of $s$ on $\Uc_h(F)$, with
the same multiplicities $\nu_u$. Hence $\lvert I_g(A_N)\rvert=N\lvert
I_g(F)\rvert$ and $\Rc^{[g]}_{\alpha,h}(A_N)=N\,\Rc^{[g]}_{\alpha,h}(F)$ by the
argument of Lemma~\ref{lem:blocks}(4), which are the last two identities because
$\lvert A_N\rvert=\lvert F\rvert N$. By Lemma~\ref{lem:blocks}(2),
$b_{h,g}(A_N)=Nb_{h,g}(F)$, which is the first. All are positive multiples of
$\lvert A_N\rvert$ because $F$ is not $B_h[g]$.
\end{proof}

At the Shannon order the same specialization gives a bound in terms of an
explicit finite ambiguity. Define the \emph{$g$-split representation ambiguity}
\begin{equation}\label{eq:ambiguity}
  \Cc_{h,g}(A):=\frac{\Lambda^{[g]}_h(A)}{n^{h}}\;\ge\;0 ,
\end{equation}
with $\Lambda^{[g]}_h$ as in \eqref{eq:LambdaF}; for $g=1$ this is
$\frac{1}{n^h}\sum_{u\in\Uc_h(A)}\nu_u\log\bigl(r_h(s(u))/\nu_u\bigr)$, the
expected logarithmic ambiguity of the representation of $s(u)$ when $u$ is drawn
from the multiset law induced by a uniformly random ordered $h$-tuple. By
Lemma~\ref{lem:blockconst}, $\Cc_{h,g}(A)$ vanishes exactly when $A$ is
$B_h[g]$.

\begin{corollary}[Entropy removal for finite sets]\label{cor:finiteshannon}
Let $A\subset G$ be finite with $\lvert A\rvert=n$ and let $g\ge1$. Then for
every $\tau\in(0,1)$,
\[
  1-\frac{b_{h,g}(A)}{n}\ \le\ \frac{\log n}{\log(1/\tau)}
   +\frac{\Cc_{h,g}(A)}{2\tau^{\,h-1}\log2} .
\]
\end{corollary}

\begin{proof}
Take $X\sim\Unif(A)$, so $H(X)=\log n$ and $\delta_{h,g}(X)=1-b_{h,g}(A)/n$. By
\eqref{eq:shannonid} and $w_u=\nu_u/n^h$ we have
$\Delta^{[g]}_{1,h}(X)=\Lambda^{[g]}_h(A)/n^h=\Cc_{h,g}(A)$. Now apply
Theorem~\ref{thm:general}(3) with $D=\log n$ and $C=\Cc_{h,g}(A)$.
\end{proof}

The bound is useful when $\log(1/\tau)$ is large compared with $\log n$ and
$\Cc_{h,g}(A)\ll\tau^{h-1}$.

\begin{remark}[channel allocation]\label{rem:signal}
Sidon-type conditions arise in optical orthogonal codes \cite{RDT}, in Golomb
rulers and the radar and sonar arrays built from them \cite{GolombGong}, and in
channel allocation \cite{TSR}. The correspondence is cleanest in the last case.
In a wavelength-division-multiplexed system a four-wave-mixing product
$f_i+f_j-f_k$ falls on a channel $f_l$ exactly when $f_i+f_j=f_k+f_l$, a
coincidence of two pair sums. A channel plan is a Sidon set precisely when it has
no coincidence $f_i+f_j=f_k+f_l$ between distinct unordered pairs of channels,
which is the case $h=2$, $g=1$; unequally spaced allocations based on Golomb
rulers were proposed for this reason \cite{TSR}. If channel $a$ is used with
probability $p_a$ and two independent uses are drawn, the order-$2$ deficit is
$\Delta_{2,2}=\log\bigl(\sum_zq_z^2/\sum_uw_u^2\bigr)$, in terms of the law $w_u$
of the unordered pair and the law $q_z$ of its sum; for uniform use it is the
logarithm of the normalized additive energy (Proposition~\ref{prop:energy}).
Theorem~\ref{thm:phase} then says when a plan whose deficit is small, and whose
input entropy is bounded at a budget order in the stable region, can be pruned
to an exactly coincidence-free plan, and at what cost in usage probability.
These are statements about the combinatorics of the plan, not about crosstalk
power.
\end{remark}

\section{Open problems}\label{sec:open}

\begin{question}\label{q:betato1}
Theorem~\ref{thm:crossover} and Remark~\ref{rem:crossunif} give the scaling
limit between the polynomial and logarithmic regimes, uniformly in absolute
error. For which growing windows $0\le x\le M(C)$ does the crossover
approximation hold with vanishing relative error, and what are the optimal
constants $c_1,c_2$ of Theorem~\ref{thm:phase}(1) at a fixed $\beta<1$?
\end{question}

\begin{question}\label{q:offdiag}
\label{q:Zfloor}
For $g=1$, Theorem~\ref{thm:phase}(4) computes $\Phi_{\alpha,\beta,D,h}$ exactly
for $0<\alpha<1$ and $\beta>\alpha$, using $h$-torsion, and
Corollary~\ref{cor:Zbeta2} computes $\Phi^{\Z}_{\alpha,2,D,h}=1-e^{-D/2}$ for
$D\le\log4$, which is strictly smaller. What is $\Phi^{\Z}_{\alpha,\beta,D,h}$
in the rest of the dilution range? For $\beta>1$ it is at least
$1-e^{-(\beta-1)D/\beta}$, and since every two-element subset of $\Z$ is $B_h$
it is at most $1-\inf\{p_{(1)}+p_{(2)}\}$ over laws with $H_\beta\le D$. The
lower bound is not attained for $1<\beta<h$ and large $D$
(Proposition~\ref{prop:Zlarge}); which integer cores are optimal for large
budgets, and is the lower bound attained when $\beta\ge h$? For $\alpha<\beta\le1$ the
heavy-atom core is unavailable, because a block of light atoms of fixed total
mass violates the budget once the block is large, and only the positive lower
bound of Corollary~\ref{cor:dilZ} is known.
\end{question}

\begin{question}
Theorem~\ref{thm:phase} treats $g$ as fixed. What happens when $g$ grows with
the block? Along the Bose--Chowla sequence of Lemma~\ref{lem:BC} one has
$M\le\lvert F\rvert^{h-1}$, so \eqref{eq:Fbound} gives
\[
  \frac{b_{h,g}(F)}{\lvert F\rvert}\ \ll_h\
  \Bigl(\frac{g}{\lvert F\rvert}\Bigr)^{1/h} ,
\]
so the mechanism behind the exact constants survives whenever
$g=o(\lvert F\rvert)$, although this counting argument does not show that
condition necessary. Does the conclusion hold in that range, or beyond it? Two
things must be checked: Lemma~\ref{lem:blockconst} is not uniform in $g$, and a
$g$ growing with $C$ must be tracked through Lemma~\ref{lem:mass}.
\end{question}

\begin{remark}\label{rem:kfibre}
The constant $d_\alpha$ in Theorem~\ref{thm:coarse} cannot be improved by
restricting the fibre profile: with $g=1$, $q_z=1$,
$w_1=w_2=\frac12(1-(k-2)\varepsilon)$ and $w_3=\dots=w_k=\varepsilon$, the defect
ratio tends to $d_\alpha$ as $\varepsilon\downarrow0$, so its infimum over fibres
with exactly $k$ elements is $d_\alpha$ for every $k\ge2$. A balance condition is
needed instead: for $k$ equal weights the ratio is
$\lvert k-k^{\alpha}\rvert/(k-1)>d_\alpha$ for $k\ge3$. The sharp constant under
the interpolating hypothesis $w_u\ge\eta q_{\pi(u)}$ is open.
\end{remark}

\section*{Acknowledgements}
The author used ChatGPT and Claude Opus for feedback on manuscript presentation,
including exposition, organization, and clarity. The author takes full
responsibility for all mathematical content and for the final manuscript.


\end{document}